\documentclass{article}
\usepackage{tikz}
\usepackage{listings}
\usepackage{graphicx} 
\usepackage{amsthm,amsfonts,amsmath,amssymb}
\usepackage{array,blkarray,multirow}

\usepackage[margin=1.2in]{geometry}

\usepackage[english]{babel}
\usepackage{hyperref}
\usepackage[small,bf,hang]{caption} 

\newtheorem{definition}{Definition}
\newtheorem{theorem}{Theorem}

\newtheorem{lemma}[theorem]{Lemma}

\newtheorem{proposition}[theorem]{Proposition}

\newtheorem{question}{Question}
\newtheorem{remark}{Remark}

\def\0{\mathbf{0}}

\usepackage{thm-restate}

\graphicspath{{figures/}}

\usepackage{color}

\newcommand{\Aut}{\operatorname{Aut}}

\begin{document}
\lstset{language=Python}          

\title{Existence of regular nut graphs for degree at most $11$}
 
\author{Patrick W. Fowler\footnote{Department of Chemistry, University of Sheffield, Sheffield, United Kingdom (p.w.fowler@sheffield.ac.uk)} ,
John Baptist Gauci\footnote{Department of Mathematics, Faculty of Science, University of Malta, Msida, Malta (john-baptist.gauci@um.edu.mt)} ,
Jan Goedgebeur\footnote{Department of Applied Mathematics, Computer Science \& Statistics, Ghent University, Ghent, Belgium (jan.goedgebeur@ugent.be)} \footnote{Computer Science Department, University of Mons, Mons, Belgium} ,\\
Toma\v{z} Pisanski\footnote{Department of Mathematics, IMFM, University of Ljubljana and University of Pri-
morska, Ljubljana, Slovenia (pisanski@upr.si)} ,  and 
Irene Sciriha\footnote{Department of Mathematics, Faculty of Science, University of Malta, Msida, Malta (irene.sciriha-aquilina@um.edu.mt)}}

\date{}

\maketitle

\centerline{Dedicated to the memory of Slobodan Simi\'{c}.}

\begin{abstract}
A nut graph is a singular graph with one-dimensional kernel and corresponding eigenvector with no zero elements.
The problem of determining the orders $n$ for which $d$-regular nut graphs exist was recently posed by Gauci, Pisanski and Sciriha. These orders are
known for $d \leq 4$. Here we solve the problem for all 
remaining cases $d\leq 11$ and determine the complete lists of all $d$-regular nut graphs of order $n$ for small values of $d$ and $n$.
The existence or non-existence of small regular nut graphs is determined by a computer search. The main tool is a construction that produces,
for any $d$-regular nut graph of order $n$, another $d$-regular nut graph of order $n + 2d$. If we are given a sufficient number of 
$d$-regular nut graphs of consecutive orders, called seed graphs, this construction may be applied in such a way that the existence of all $d$-regular
nut graphs of higher orders is established. For even $d$ the orders $n$ are indeed consecutive, while for odd $d$ the orders $n$ are consecutive even numbers.
Furthermore, necessary conditions for combinations of order and degree for vertex-transitive nut graphs are derived.
\end{abstract}

\section{Introduction} \label{Intro}

In this paper we consider graphs that are \emph{simple}, that is, without loops or multiple edges. Given a labelled graph $G$ of order $n$, its $0-1$ \emph{adjacency matrix} ${\bf A}={\bf A}(G)$ is the $n \times n$ matrix with entries $a_{ij}$ (for $i,j\in\{1,\ldots,n\}$) such that $a_{ij}=1$ if and only if there is an edge between the vertices $i$ and $j$ in $G$. If ${\bf A}(G)$ has an eigenvalue zero, the graph $G$ is \emph{singular} with nullity $\eta=\eta(G)$ equal to the multiplicity of the eigenvalue zero of ${\bf A}(G)$. In the sequel, we use terminology for a graph and its adjacency matrix interchangeably. A graph that has an all-non-zero eigenvector corresponding to the zero eigenvalue is a \emph{core graph}. If, in addition, the dimension of the
nullspace is one, the corresponding graph is a \emph{nut graph}~\cite{sciriha1998construction,
ScOnRnkGr99,Sciriha2007characterisation,Sciriha2008Coalesced,ScirihaGutman-NutExt}. 
Some authors consider the isolated vertex to be a trivial case of a nut graph, and would consider nuts proper, or 
\emph{non-trivial} nuts, to have $n > 1$. In the present paper, the only nut graphs that are considered are non-trivial, in this sense.

Two chemical motivations can be claimed for the study of nut graphs
of maximum degree $\leq 3$ (the \emph{chemical} nut graphs). First, the eigenvalue
zero is special in a chemical context, as it corresponds in H{\"u}ckel Molecular Orbital Theory
to the non-bonding energy level of a conjugated carbon $\pi$ system, from which electrons are 
easily removed and to which they are easily added.
If this happens to be the highest occupied molecular
orbital, then partial occupation will correspond to a radical
(molecule with unpaired spins) in which spin density and hence radical reactivity is distributed 
over all carbon atoms of the framework. 
Secondly, it has been shown~\cite{fowler2014omni} that nut graphs correspond exactly to the classes of distinct
and strong omni--conductors of nullity one, according to the simplest
model of molecular conduction.
(A \textit{distinct} omni--conductor is a molecule for which connection in a circuit via any distinct pair of atoms gives transmission; for a \textit{strong} omni--conductor, the transmission occurs whether the connection atoms are distinct or not.)

Using a computer search, Fowler, Pickup, Todorova, Borg and Sciriha~\cite{fowler2014omni} determined all nut graphs up to 10 vertices, and all chemical nut graphs up to 16 vertices. Later Coolsaet, Fowler and Goedgebeur~\cite{CoolFowlGoed-2017} designed a specialised algorithm to generate nut graphs and used it to determine all nut graphs up to 13 vertices and all chemical nut graphs up to 22 vertices.
We refer the reader to~\cite{CoolFowlGoed-2017} and~\cite{GPS} for a more detailed overview.

In~\cite{ScirihaGutman-NutExt} Sciriha and Gutman proved that the smallest non-trivial nut graphs have order 7 and that nut graphs exist for all orders at least 7.
A natural question arises: 
\begin{question}\label{question:smallest_order}
What is the smallest order for which a regular nut graph exists?
\end{question}
A chemical motivation for the study of regular nut graphs is that
$3$-regular, $3$-connected graphs are candidates as carbon cages,
so that nut graphs of this type would themselves be candidates for 
omni-conductors and radical
molecules of the kinds mentioned above. 

The two main results of~\cite{GPS} can be stated as follows.

\begin{theorem} \label{CubicNut}
Cubic nut graphs on $n$ vertices exist if and only if $n$ is an even integer, $n \geq 12$ and $n \notin \{14,16\}$.
\end{theorem}

\begin{theorem} \label{QuarticNut}
Quartic nut graphs on $n$ vertices exist if and only if $n=8,10,12$ or $n\geq 14$.
\end{theorem}

The results of a computer search presented in Figure~\ref{fig:orders_small_nuts} have indeed shown that the answer to the above question is the following.
\begin{remark}[Answer to Question~\ref{question:smallest_order}]
The smallest order for which a regular nut graph exists is 8.
\end{remark}
Since no nut $d$-regular graph exists for $d = 1$ or $d=2$, the above results raise a number of related questions. (See also Section~\ref{sect:conclusion} ``Conclusion'').

\begin{question}
\textit{Is it true, that for each $d > 2$ there exists a $d$-regular nut graph?}
\end{question}

In~\cite{GPS}, the set $N(d)$ was defined as the set consisting of all integers $n$ for which a $d$-regular nut graph
of order $n$ exists. Therefore we have $N(1) = N(2) = \emptyset$, $N(3) = \{12\} \cup \{2k: k\geq 9\}$,  $N(4) = \{8,10,12\} \cup \{k : k\geq14\}$.

\begin{definition}
Let us call a pair \emph{$(n,d)$ admissible} if there exists a $d$-regular (simple) graph of order $n$ and let us denote by $A(d)$ the collection of integers $n$, such that there exists an $(n,d)$-admissible graph. Let us call a pair $(n,d)$ \emph{nut-realisable} if and only if there exists a $d$-regular nut graph of order~$n$. 
\end{definition}
For $d$ even, $(n,d)$ is admissible if and only if
$d < n$. For $d$ odd, $(n,d)$ is admissible if and only if $d < n$ and $n$ is even. Clearly $N(d) \subset A(d)$.
A valence $d$ for which the set $A(d) \setminus N(d)$ is finite will be called \emph{normal}.

\begin{question}
\textit{Is it true, that for each $d > 2$ there are only a finite number of admissible orders that do not belong to $N(d)$?} This is equivalent to saying, is the complement $A(d) \setminus N(d)$ finite? In other words, is it true that $d$ is normal if and only if $d > 2$?
\end{question}

In this note, we determine $N(d)$, and give positive answer to the above question, for every $d \leq 11$. That is:

\begin{restatable}{theorem}{mainthm}\label{thm_orders}
The following holds:
\begin{itemize}
\item $N(5) = \{2k: k\geq 5\}$
\item $N(6) = \{k: k\geq 12\}$
\item $N(7) = \{2k: k\geq 6\}$
\item $N(8) = \{12\} \cup \{k : k\geq 14\}$
\item $N(9) = \{2k: k\geq 8\}$
\item $N(10) = \{k: k\geq 15\}$
\item $N(11) = \{2k: k\geq 8\}$
\end{itemize}
\end{restatable}

We use a construction that was proved and applied repeatedly in~\cite{GPS}. It extends any singular graph $G$ at a vertex $v$ of valency $d$, to give a graph $F(G,v)$ with the same nullity as $G$ but with $2d$ more vertices (each of degree $d$), whilst all other vertices retain the degree they had in $G$. The process used in this extension, named in~\cite{GPS} the \emph{Fowler Construction}, is depicted in Figure~\ref{Fig-FowlerExt}.

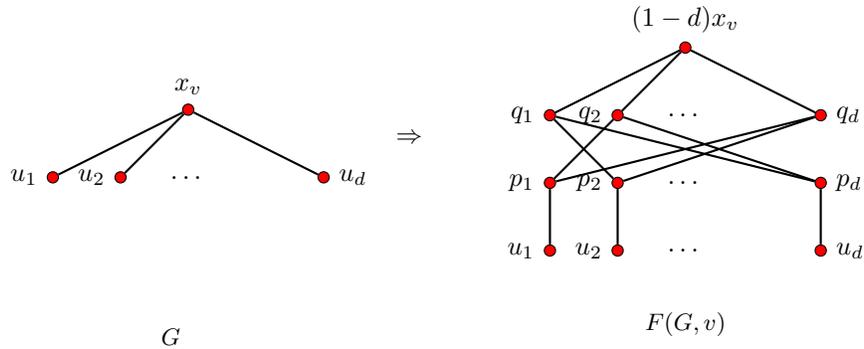
\begin{figure}[h]
\begin{center}
\begin{tabular}{ >{\centering\arraybackslash} m{4.5cm} >{\centering\arraybackslash} m{1cm} >{\centering\arraybackslash} m{5.5cm} }
\begin{tikzpicture}[scale=0.9, every edge/.style = {draw, thick},
                    vertex/.style args = {#1 #2}{circle,
                                                draw, fill=red, inner sep=1.5pt,
                                                label=#1:#2}]
\path	node(v_1) [vertex=above $x_v$] at (0,0) {}
	node(u_1) [vertex=left $u_1$] at (-2,-1) {}
	node(u_2) [vertex=left $u_2$] at (-1,-1) {}
	node(u_d) [vertex=right $u_d$] at (2,-1) {}
	(v_1) edge (u_1)
	(v_1) edge (u_2)
	(v_1) edge (u_d);

 \path (u_1) -- node[auto=false]{\ldots} (u_d);
\end{tikzpicture}

&

$\Rightarrow$

&

\begin{tikzpicture}[scale=0.9, every edge/.style = {draw, thick},
                    vertex/.style args = {#1 #2}{circle,
                                                draw, fill=red, inner sep=1.5pt,
                                                label=#1:#2}]
\path	node(v_1) [vertex=above $(1-d)x_v$] at (0,0) {}
	node(q_1) [vertex=left $q_1$] at (-2,-1) {}
	node(q_2) [vertex=left $q_2$] at (-1,-1) {}
	node(q_d) [vertex=right $q_d$] at (2,-1) {}
	node(p_1) [vertex=left $p_1$] at (-2,-2) {}
	node(p_2) [vertex=left $p_2$] at (-1,-2) {}
	node(p_d) [vertex=right $p_d$] at (2,-2) {}
	node(u_1) [vertex=left $u_1$] at (-2,-3) {}
	node(u_2) [vertex=left $u_2$] at (-1,-3) {}
	node(u_d) [vertex=right $u_d$] at (2,-3) {}
	(v_1) edge (q_1)
	(v_1) edge (q_2)
	(v_1) edge (q_d)
	(q_1) edge (p_2)
	(q_1) edge (p_d)
	(q_2) edge (p_1)
	(q_2) edge (p_d)
	(q_d) edge (p_1)
	(q_d) edge (p_2)
	(p_1) edge (u_1)
	(p_2) edge (u_2)
	(p_d) edge (u_d);

 \path (q_1) -- node[auto=false]{\ldots} (q_d);
 \path (p_1) -- node[auto=false]{\ldots} (p_d);
 \path (u_1) -- node[auto=false]{\ldots} (u_d);
\end{tikzpicture}

\\

\small{$G$} & & \small{$F(G,v)$}

\end{tabular}
\end{center}
\vspace{-15pt}\caption{The {\lq Fowler Construction\rq}~\cite{GPS}.
Notation: 
$x_v$, $u_i$ are entries in the kernel 
eigenvector of $G$, with entries $u_i$ summing to zero;
in $F(G,v)$, the two layers of extra vertices span a 
complete bipartite graph minus a perfect matching, 
with eigenvector entries
$q_i = u_i$, $p_i = x_v$, and the entry on vertex $v$ 
replaced by $(1-d)x_v$. 
} \label{Fig-FowlerExt}
\end{figure}

From the construction, it follows immediately that:

\begin{lemma} \label{ExistRegNuts}
Let $G$ be a nut graph on $n$ vertices and let $v$ be a vertex of $G$ having valency $d$. Then there exists a nut graph $F(G,v)$ on $n+2d$ vertices in which the $2d$ new vertices all have valency $d$.
\end{lemma}

Starting from a nut graph of order $n$, the above lemma can be used repeatedly to produce an infinite sequence of nut graphs on $n, n+2d, n+4d,n+6d, \ldots $ vertices in which all of the newly introduced vertices have valency $d$. In particular, this argument can be applied to any $d$-regular nut graph on $n$ vertices to obtain an infinite sequence of $d$-regular nut graphs.

For even valency $d$, one has to establish the existence of $d$-regular nut graphs for $2d$ consecutive orders, starting with order $n$. By Lemma~\ref{ExistRegNuts} above, this would then imply the existence of $d$-regular nut graphs of any order at least $n$. Smaller cases have to be checked by computer search or some other method. In the case of $d=4$, such a run of eight consecutive orders was obtained starting at $n = 14$~\cite{GPS}

By the well-known Handshaking Lemma, for odd valency $d$ the order of a graph must be even. It therefore suffices to establish existence of $d$-regular nut graphs for $d$ consecutive even orders, starting with order $n$. For $d = 3$ the run of three even orders started at $n = 18$~\cite{GPS}.


\section{Existence of regular nut graphs} \label{sect:regnuts}

In~\cite{CoolFowlGoed-2017}, a generation algorithm for nut graphs is presented. We could in principle use this generator and restrict the generation to $d$-regular graphs.
However, it turned out to be more efficient to use Meringer's generator for regular graphs,
\textit{genreg}~\cite{meringer_99},  
to generate $d$-regular graphs and then test which 
are nut graphs (using the filter program from~\cite{nutgen-site}).

Using this approach, we determined all 5-regular nut graphs up to 18 vertices. The counts can be found in Table~\ref{Table-AllShort}. Using a similar approach (but using the generator from~\cite{cubicpaper} for the cubic case), we also determined all 4-regular nut graphs up to 15 vertices and all 3-regular nut graphs up to 22 vertices, 
the counts for which can be found in Table~\ref{Table-AllShort}. These graphs can also be downloaded from the \textit{House of Graphs}~\cite{hog} at
\url{https://hog.grinvin.org/Nuts}. 

We now proceed to the statement and proof of our main result.

\mainthm*

\begin{proof}
It follows from Table~\ref{Table-AllShort} that the smallest 5-regular nut graphs have 10 vertices. In order to cover all orders using the 
construction, we need to present a 5-regular nut graph of every even order $10 \leq n \leq 18$. 
In the Appendix we give an example of a 5-regular nut graph of every such order $n$. 
In fact, the examples from the Appendix are chosen to exhibit the largest automorphism group size amongst the 5-regular nut graphs of that order.

The proof for the other cases of $N(d)$ for $6 \leq d \leq 11$ is completely analogous. That is: using Meringer's generator \textit{genreg}~\cite{meringer_99} together with a 
filter which tests whether
a given graph is a nut, we determined the smallest order for which a $d$-regular nut graph exists and excluded small orders for which $d$-regular 
nut graphs do not exist (e.g.\ to prove the nonexistence of an 8-regular nut graph on 13 vertices). 
See Table~\ref{Table-AllShort} for the counts of the smallest $d$-regular nuts 
for $6 \leq d \leq 11$.
We have used \textit{genreg} to establish the existence of a $d$-regular nut graph of order $n$ for $2d$ 
consecutive orders (if $d$ is even), or the existence of a $d$-regular nut graph for $d$ consecutive 
even orders (if $d$ is odd) which are then used as seeds for the 
construction. 

The adjacency list of a $d$-regular seed graph of every required order can be found in the Appendix. 
These graphs can also be inspected in the database of interesting graphs from the \textit{House of Graphs}~\cite{hog} 
by searching with keywords ``d-regular nut graph'' (with $d \in \{5, \ldots,11\}$).
\end{proof}

\begin{table}[htb!] \centering
\begin{tabular}{| >{\centering\arraybackslash}m{1cm} | >{\centering\arraybackslash}m{2cm} || >{\centering\arraybackslash}m{1cm} | >{\centering\arraybackslash}m{2cm} ||>{\centering\arraybackslash}m{1cm} |>{\centering\arraybackslash}m{2cm} |}
\hline
\multicolumn{2}{|c||}{3-regular nut graphs} & \multicolumn{2}{c||}{4-regular nut graphs} & \multicolumn{2}{c|}{5-regular nut graphs}\\
\hline \hline
Order & Number of graphs & Order & Number of graphs & Order & Number of graphs\\
\hline \hline
12 & 9 & 8 & 1 & 10 & 9\\ \hline
18 & 5 541 & 10 & 12 & 12 & 4 \\ \hline
20 & 5 & 12 & 269 & 14 & 25 \\ \hline
22 & 71 & 14 & 15 633 & 16 & 13 530 \\ \hline
 &  & 15 & 1 & 18 & 665 456 900 \\ \hline
\end{tabular}

\medskip

\begin{tabular}{| >{\centering\arraybackslash}m{1cm} | >{\centering\arraybackslash}m{2cm} || >{\centering\arraybackslash}m{1cm} | >{\centering\arraybackslash}m{2cm} ||>{\centering\arraybackslash}m{1cm} |>{\centering\arraybackslash}m{2cm} |}
\hline
\multicolumn{2}{|c||}{6-regular nut graphs} & \multicolumn{2}{c||}{7-regular nut graphs} & \multicolumn{2}{c|}{8-regular nut graphs}\\
\hline \hline
Order & Number of graphs & Order & Number of graphs & Order & Number of graphs\\
\hline \hline
12 & 1 964 &  12 & 3         & 12 & 24\\ \hline
13 & 79 &     14 & 5 168 453           & 13 & 0 \\ \hline
14 & 1 872 &   &  & 14 & 424 088 \\ \hline
\end{tabular}

\medskip

\begin{tabular}{| >{\centering\arraybackslash}m{1cm} | >{\centering\arraybackslash}m{2cm} || >{\centering\arraybackslash}m{1cm} | >{\centering\arraybackslash}m{2cm} ||>{\centering\arraybackslash}m{1cm} |>{\centering\arraybackslash}m{2cm} |}
\hline
\multicolumn{2}{|c||}{9-regular nut graphs} & \multicolumn{2}{c||}{10-regular nut graphs} & \multicolumn{2}{c|}{11-regular nut graphs}\\
\hline \hline
Order & Number of graphs & Order & Number of graphs & Order & Number of graphs\\
\hline \hline
14 & 0 &  	  			14 & 0         & 14 & 0\\ \hline
16 & $>0$ &     				15 & 173 650   & 16 & 3 316 \\ \hline
\end{tabular}

\caption{The numbers of $d$-regular nut graphs for $d \in \{3, \ldots, 11\}$. No such graphs exist for smaller orders.
}\label{Table-AllShort}
\end{table}

\section{Vertex-transitive nut graphs}
A graph $G$ such that all vertices are equivalent under the 
automorphism group $\Aut(G)$ is  
\emph{vertex-transitive}.
Requiring regular nut graphs to be vertex-transitive clearly imposes further restrictions, and it seems
natural to ask the following question.

\begin{question}
For what pairs $(n, d)$ does a vertex-transitive nut graph of order $n$ and degree $d$ exist?
\end{question}

Some straightforward observations give the following necessary
conditions for the existence of a vertex-transitive nut graph, as summarised in the
following theorem:

\begin{theorem} \label{thm:vtnut}
Let $G$ be a vertex-transitive nut graph
on $n$ vertices, of degree $d$. 
Then $n$ and $d$ satisfy the following conditions.
\emph{Either}
$d = 0 \, \mathrm{mod}\, 4$, 
and $n = 0\, \mathrm{mod} \,2$ and  $n \geq d+4$;
\emph{or}
$d = 2 \,\mathrm{mod}\,4$, 
and $n = 0 \, \mathrm{mod} \,4$ and  $n \geq d+6$.
\end{theorem}

\begin{proof}
We are going to prove the result in five steps. 

Let $\mathbf{x}$ be a kernel eigenvector of a vertex-transitive nut graph $G$. 
As the zero eigenvalue is simple, $\mathbf{x}$ transforms as a non-degenerate
irreducible representation of $\Aut(G)$. Hence, 
the eigenvector can be chosen to have all entries real.
Under any given automorphism, each eigenvector 
entry $x_i$ transforms to $x_j$ or $-x_j$ for some $j$, 
and therefore all squared entries $x_i\sp2$ are equal. 
We can choose a normalisation where each entry in $\mathbf{x}$ is either $+1$ or $-1$.

A kernel eigenvector obeys the local adjacency condition:
\begin{equation}
\sum_{i \sim j} x_{j} = 0
\end{equation}
for every vertex $i$, which implies that the neighbourhood of every
vertex of a vertex-transitive nut graph contains an equal number 
of vertices bearing entries $+1$ and $-1$
in $\mathbf{x}$. Hence, a vertex-transitive nut graph is of even degree.

Next, note that a kernel eigenvector $\mathbf{x}$ of a $d$-regular graph obeys

\begin{equation}  
\sum_{i=1}^n x_{i} =  0,
\end{equation}
as $\mathbf{x}$ is orthogonal to the all-ones Perron eigenvector that corresponds to the 
maximum eigenvalue $\lambda_{\mathrm{max}} = d$ of a regular graph.
Hence, for a vertex-transitive nut graph $G$, the sum over 
all of the $+1$ and $-1$ entries $x_i$ is zero, from which we deduce that $n$ is even.

Consider the subgraphs of $G$ induced by the vertices with entries $x_i = +1$
and $x_i =-1$, respectively. 
Let $H$ be the subgraph induced by the vertices with $x_i = +1$.
Vertex $i$ in $G$ is of even degree, and in the neighbourhood of $i$ in $G$ 
there are $d/2$ vertices carrying entries of each sign.
The graph $H$ is therefore regular, of degree $d/2$ and of order $n/2$. 
Hence, if $d = 2\, \mathrm{mod}\, 4$, $d/2$ is odd and the order of $H$ must be even. 
Therefore, $n/2$ is even and we have $n = 0\, \mathrm{mod}\, 4$ for 
$d = 2\, \mathrm{mod}\, 4$.

Finally, note a limitation on the smallest vertex-transitive nut graph of given degree $d$. 
A $d$-regular graph has at least $(d+1)$ vertices. Given that both degree and order 
of a vertex-transitive nut graph are even, the order of a vertex-transitive nut graph
obeys $n \geq d+2$. However, a vertex-transitive nut graph with $n = d+2$ is not possible.
For even $n$ and $d$, there is a unique vertex-transitive graph, $H\sp\prime$, of 
order $d+2$~\cite{mckay1990transitive}, which is not a nut graph. 
Indeed, as every vertex of $H\sp\prime$ is adjacent to all but one of the others, 
we can choose a labelling such that vertices
$i$ and $d+1-i$ of $H\sp\prime$ are duplicates for all $0 \leq i \leq n-1$.
In particular, $0$ and $d+1$ have the same neighbourhood, $\{1, \ldots , d\}$.
The subgraph induced by this neighbourhood is $K_d$ minus a perfect matching, 
as edges between $j$ and $d+1-j$ are missing for $1 \leq j \leq d$.
This is the well known cocktail-party graph on
$d$ vertices, $CP(d/2)$. It has $d/2$ distinct pairs of 
non-adjacent vertices.
A kernel eigenvector for $H'$ can be constructed as follows: 
assign values $+1$ to vertex $0$, $-1$ to vertex $d+1$, and zero to all others.
Hence, $H\sp\prime$ is not a nut graph, and the stated limits on $n$ follow. 
\end{proof}

Calculations using the catalogue of vertex-transitive graphs created by Holt and Royle~\cite{HOLT2019} 
indicate that \emph{all} pairs $(n, d)$ compatible with the above conditions and with $n \leq 42$ 
correspond to at least one vertex-transitive nut graph. 

See, for example, the blue entries in Figure~\ref{fig:orders_small_nuts}. This figure gives an overview of the existence of $d$-regular and vertex-transitive nut graphs of order $n \leq 30$. 

Our question is therefore refined to:

\begin{question}
For what pairs $(n, d)$ obeying Theorem~\ref{thm:vtnut} 
does a vertex-transitive nut graph of order $n$ and degree $d$ exist?
(In other words: what pairs are \emph{realisable}, in the sense that a pair $(n, d)$ is realisable if there
exists at least one vertex-transitive nut graph with these parameters?)
\end{question}

We do know that there is at least one infinite family of realisable pairs $(n, d)$ of vertex-transitive graphs,
as all antiprisms with largest ring size not divisible by $3$ are nut graphs;
hence all pairs $(n, 4)$ with $n \ne 0\, \mathrm{mod} \,6$ are vertex-transitive-nut-realisable.

\begin{figure}[!htb]
\begin{center}
 \includegraphics[width=1.0\textwidth, trim={20mm 35mm 30mm 25mm},clip]{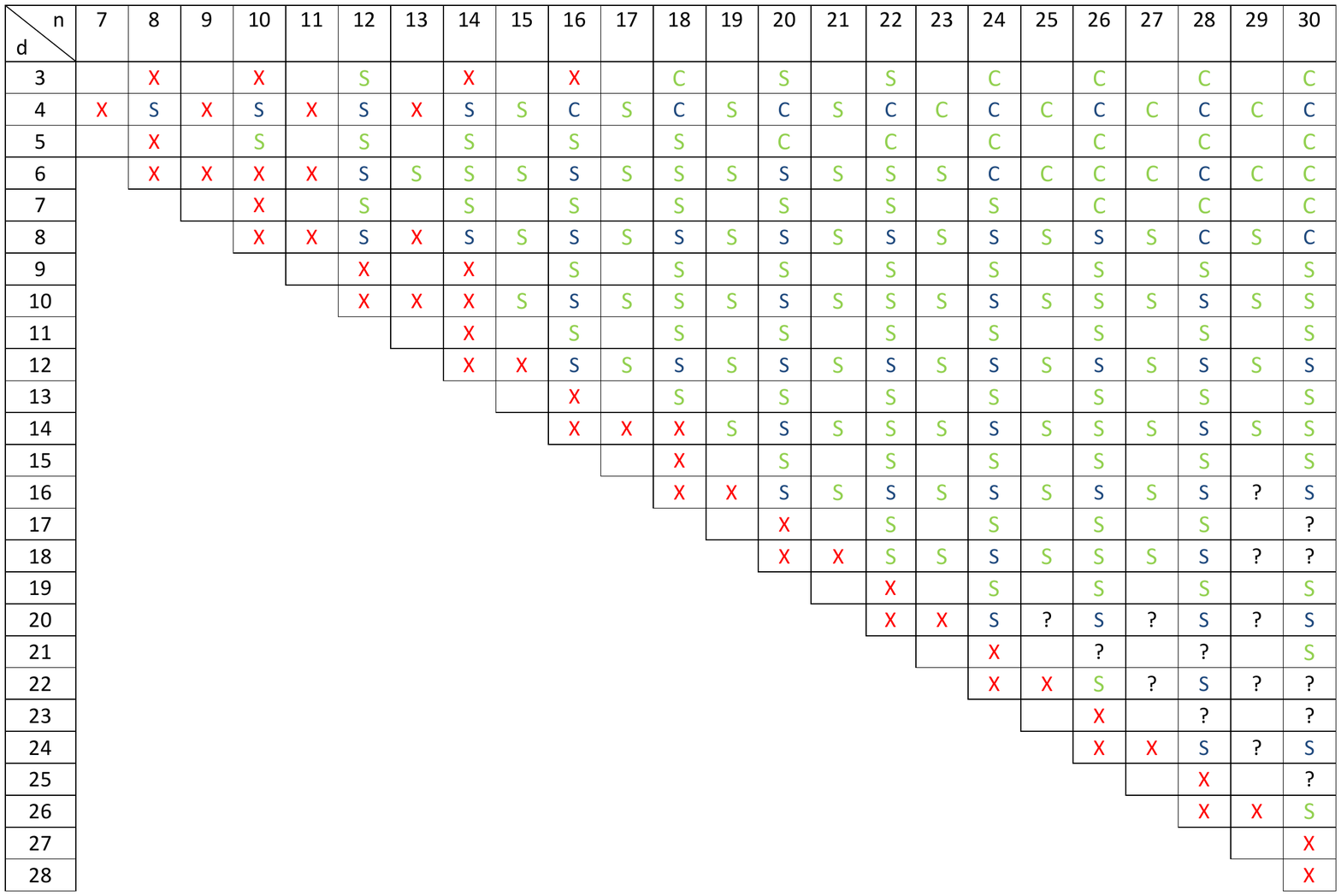}
	\caption{Snapshot of small regular nut graphs: calculations and proven results for order $n$ and degree $d$ for $n \leq 30$. 
Symbols: blank spaces and red crosses indicate non-existence; question marks indicate open cases; the symbol S indicates a seed graph; the symbol C indicates that at least one example of a regular nut graph can be obtained by the Fowler construction; blue (resp., green) symbols indicate the existence (resp., non-existence) of a vertex-transitive nut graph.}
	
	\label{fig:orders_small_nuts}
\end{center}
\end{figure}


\section{Conclusion}
\label{sect:conclusion}

We divide the set of nut graphs into two disjoint subsets. If a graph can be obtained from the Fowler construction,
we call it a \textit{C-graph}. If it cannot be obtained in this way, we call it a seed graph or \textit{S-graph}.
This carries over to 
nut-realisable pairs $(n,d)$. Such a pair is a \textit{C-pair} if there exists at least one $(n,d)$-graph that is a C-graph. On the other hand, if all $(n,d)$-graphs are seed graphs, we call the pair $(n,d)$ an \textit{S-pair} or \textit{seed pair}.

\begin{proposition}
If $d$ is a normal valence, then the following is true. 
If $d$ is odd there are exactly $d$ seed pairs $(n,d)$ and if $d$ is even, there are exactly $2d$ such seed pairs.
In the former case each seed value of $n$ has a different (even) remainder mod $2d$ while in the latter case each $n$ gives a different remainder mod $2d$.
\end{proposition}

Let $d$ be a normal valence. We can define two numbers: $n_0(d)$ and $n_1(d)$. Here $n_0(d)$ represents the
smallest value of $n$ such that the pair $(n,d)$ is nut-realisable. On the other hand let $n_1(d)$ represent the smallest
value such that admissible pairs $(n,d)$ with $n \geq n_1(d)$ are nut-realisable.

Our results from Figure~\ref{fig:orders_small_nuts} lead to the following values of $n_0(d)$ and $n_1(d)$ in Table~\ref{table:n0_n1}.

\begin{table}[htb!] 
\centering
\begin{tabular}{r|rr}
$d$ & $n_0(d)$ & $n_1(d)$ \\
\hline
3 & 12 & 18 \\
4 & 8 & 14 \\
5 & 10 & 10 \\
6 & 12 & 12 \\
7 & 12 & 12 \\
8 & 12 & 14 \\
9 & 16 & 16 \\
10 & 15 & 15 \\
11 & 16 & 16 \\
12 &  16 & $\geq 16$ \\
13& 18 & $\geq 18$ \\
14& 19 & $\geq 19$ \\
15&  20 & $\geq 20$ \\
16&  20 & $\geq 20$ \\
17&  22 & $\geq 22$ \\
18&  22 & $\geq 22$ \\
19& 24 &   $\geq 24$ \\
20&  $24$ &  $\geq 24$ \\
21&  $\geq 26$ &  $\geq 26$ \\
22&  $26$ &  $\geq 26$ \\
23&  $\geq 28$ &  $\geq 28$ \\
24&  $28$ &  $\geq 28$ \\
25&  $\geq 30$ &  $\geq 30$ \\
26&  $30$ &  $\geq 30$ \\
 \end{tabular}  

 	\caption{Known values of $n_0(d)$ and $n_1(d)$. An entry of the form $n_0(d) \geq p$, indicates that a complete search of all feasible smaller cases of order less than $p$ has not revealed an example of a $d$-regular nut graph.
 	}
 	 	\label{table:n0_n1}
\end{table}

Variants of the original question suggest themselves.
For example: 

\begin{question}
Determine the set $N(p,g)$ such that there exists a $p$-regular nut graph of 
order $n$ and and girth (at least) $g$ if and only if $n \in N(p,g)$.
\end{question}

\begin{question}
Determine the set $N(p)$ such that there exists a $q$-regular nut graph of 
order $n$ for all degrees $3 \le q \le p$ if and only if $n \in N(p)$.
\end{question}



\section*{Acknowledgements}
Jan Goedgebeur is supported by a Postdoctoral Fellowship of the Research Foundation Flanders (FWO).
The work of Toma\v{z} Pisanski is supported in part by the Slovenian Research Agency (research program P1-0294 
and research projects N1-0032, J1-9187, J1-1690), 
and in part by H2020 Teaming InnoRenew CoE.
Several computations for this work were carried out using the Stevin Supercomputer Infrastructure at Ghent University.




\bibliographystyle{plain}
\bibliography{references.bib}


\section*{Appendix}

\subsection*{Adjacency lists of regular nut graphs}

The graphs in this section act as seed graphs in the proof of Theorem~\ref{thm_orders}. For completeness we also list the seed graphs for $d= 3$ and $d=4$.
These graphs can also be inspected at the database of interesting graphs from the 
\textit{House of Graphs}~\cite{hog} by searching for the keywords ``d-regular nut graph'' (with $d \in \{3,...,11\}$).

%

\subsubsection*{Adjacency lists of 3-regular nut graphs}

\begin{itemize}
\addtolength{\itemsep}{-2mm}
\footnotesize

\item Order 12: \{0: 1 2 3; 1: 0 4 5; 2: 0 6 7; 3: 0 8 9; 4: 1 6 8; 5: 1 7 9; 6: 2 4 10; 7: 2 5 11; 8: 3 4 11; 9: 3 5 10; 10: 6 9 11; 11: 7 8 10\}

\item Order 20: \{0: 1 2 3; 1: 0 4 5; 2: 0 6 7; 3: 0 8 9; 4: 1 6 10; 5: 1 8 11; 6: 2 4 12; 7: 2 11 13; 8: 3 5 14; 9: 3 12 15; 10: 4 13 16; 11: 5 7 17; 12: 6 9 16; 13: 7 10 18; 14: 8 15 19; 15: 9 14 18; 16: 10 12 19; 17: 11 18 19; 18: 13 15 17; 19: 14 16 17\}

\item Order 22: \{0: 9 14 15; 1: 10 16 18; 2: 11 13 14; 3: 11 13 15; 4: 12 18 21; 5: 12 20 21; 6: 15 16 17; 7: 16 19 20; 8: 17 19 20; 9: 0 14 19; 10: 1 17 18; 11: 2 3 13; 12: 4 5 21; 13: 2 3 11; 14: 0 2 9; 15: 0 3 6; 16: 1 6 7; 17: 6 8 10; 18: 1 4 10; 19: 7 8 9; 20: 5 7 8; 21: 4 5 12\}

\end{itemize}

\subsubsection*{Adjacency lists of 4-regular nut graphs}

\begin{itemize}
\addtolength{\itemsep}{-2mm}
\footnotesize

\item Order 8: \{0: 1 2 3 4; 1: 0 2 3 5; 2: 0 1 4 6; 3: 0 1 5 7; 4: 0 2 6 7; 5: 1 3 6 7; 6: 2 4 5 7; 7: 3 4 5 6\}

\item Order 10: \{0: 1 2 3 4; 1: 0 2 3 5; 2: 0 1 4 6; 3: 0 1 5 7; 4: 0 2 6 8; 5: 1 3 7 9; 6: 2 4 8 9; 7: 3 5 8 9; 8: 4 6 7 9; 9: 5 6 7 8\}

\item Order 12: \{0: 1 2 3 4; 1: 0 2 3 4; 2: 0 1 3 5; 3: 0 1 2 5; 4: 0 1 5 6; 5: 2 3 4 7; 6: 4 7 8 9; 7: 5 6 10 11; 8: 6 9 10 11; 9: 6 8 10 11; 10: 7 8 9 11; 11: 7 8 9 10\}

\item Order 14: \{0: 1 2 3 4; 1: 0 2 3 4; 2: 0 1 3 5; 3: 0 1 2 5; 4: 0 1 5 6; 5: 2 3 4 7; 6: 4 8 9 10; 7: 5 8 9 10; 8: 6 7 11 12; 9: 6 7 11 13; 10: 6 7 12 13; 11: 8 9 12 13; 12: 8 10 11 13; 13: 9 10 11 12\}

\item Order 15: \{0: 1 2 3 4; 1: 0 2 3 5; 2: 0 1 6 7; 3: 0 1 6 8; 4: 0 7 9 10; 5: 1 9 11 12; 6: 2 3 10 11; 7: 2 4 13 14; 8: 3 9 13 14; 9: 4 5 8 11; 10: 4 6 12 13; 11: 5 6 9 14; 12: 5 10 13 14; 13: 7 8 10 12; 14: 7 8 11 12\}

\item Order 17: \{0: 1 2 3 4; 1: 0 2 3 4; 2: 0 1 3 5; 3: 0 1 2 6; 4: 0 1 5 7; 5: 2 4 6 8; 6: 3 5 9 10; 7: 4 9 11 12; 8: 5 11 13 14; 9: 6 7 10 13; 10: 6 9 14 15; 11: 7 8 12 16; 12: 7 11 13 16; 13: 8 9 12 15; 14: 8 10 15 16; 15: 10 13 14 16; 16: 11 12 14 15\}

\item Order 19: \{0: 1 2 3 4; 1: 0 2 3 4; 2: 0 1 3 4; 3: 0 1 2 5; 4: 0 1 2 6; 5: 3 6 7 8; 6: 4 5 7 9; 7: 5 6 8 10; 8: 5 7 9 11; 9: 6 8 11 12; 10: 7 11 13 14; 11: 8 9 10 13; 12: 9 13 15 16; 13: 10 11 12 14; 14: 10 13 17 18; 15: 12 16 17 18; 16: 12 15 17 18; 17: 14 15 16 18; 18: 14 15 16 17\}

\item Order 21: \{0: 1 2 3 4; 1: 0 2 3 5; 2: 0 1 6 7; 3: 0 1 8 9; 4: 0 6 10 11; 5: 1 10 12 13; 6: 2 4 14 15; 7: 2 12 14 16; 8: 3 12 16 17; 9: 3 13 14 18; 10: 4 5 16 19; 11: 4 15 18 19; 12: 5 7 8 15; 13: 5 9 18 20; 14: 6 7 9 18; 15: 6 11 12 20; 16: 7 8 10 17; 17: 8 16 19 20; 18: 9 11 13 14; 19: 10 11 17 20; 20: 13 15 17 19\}

\end{itemize}

\subsubsection*{Adjacency lists of 5-regular nut graphs}

\begin{itemize}
\addtolength{\itemsep}{-2mm}
\footnotesize

\item Order 10: \{0: 1 2 3 4 5;  1: 0 2 3 6 7;  2: 0 1 4 6 8;  3: 0 1 4 7 9;  4: 0 2 3 8 9;  5: 0 6 7 8 9;  6: 1 2 5 7 8;  7: 1 3 5 6 9;  8: 2 4 5 6 9;  9: 3 4 5 7 8\}

\item Order 12: \{0: 1 2 3 4 5;  1: 0 2 3 4 6;  2: 0 1 5 7 8;  3: 0 1 6 9 10;  4: 0 1 7 9 11;  5: 0 2 8 9 10;  6: 1 3 7 8 10;  7: 2 4 6 9 11;  8: 2 5 6 10 11;  9: 3 4 5 7 11; 10: 3 5 6 8 11; 11: 4 7 8 9 10\}

\item Order 14: \{0: 1 2 3 4 5;  1: 0 2 3 4 5;  2: 0 1 3 4 6;  3: 0 1 2 7 8;  4: 0 1 2 9 10;  5: 0 1 6 11 12;  6: 2 5 7 8 9;  7: 3 6 8 9 13;  8: 3 6 7 11 12;  9: 4 6 7 10 13; 10: 4 9 11 12 13; 11: 5 8 10 12 13; 12: 5 8 10 11 13; 13: 7 9 10 11 12\}

\item Order 16: \{0: 1 2 3 4 5;  1: 0 2 3 4 5;  2: 0 1 3 4 5;  3: 0 1 2 4 6;  4: 0 1 2 3 7;  5: 0 1 2 8 9;  6: 3 7 8 10 11;  7: 4 6 12 13 14;  8: 5 6 9 10 12;  9: 5 8 10 12 15; 10: 6 8 9 11 15; 11: 6 10 13 14 15; 12: 7 8 9 13 14; 13: 7 11 12 14 15; 14: 7 11 12 13 15; 15: 9 10 11 13 14\}

\item Order 18: \{0: 1 2 3 4 5;  1: 0 2 3 4 5;  2: 0 1 3 4 5;  3: 0 1 2 4 6;  4: 0 1 2 3 6;  5: 0 1 2 7 8;  6: 3 4 9 10 11;  7: 5 9 10 11 12;  8: 5 9 10 13 14;  9: 6 7 8 10 13; 10: 6 7 8 9 13; 11: 6 7 12 15 16; 12: 7 11 15 16 17; 13: 8 9 10 14 17; 14: 8 13 15 16 17; 15: 11 12 14 16 17; 16: 11 12 14 15 17; 17: 12 13 14 15 16\}
\end{itemize}

\subsubsection*{Adjacency lists of 6-regular nut graphs}

\begin{itemize}
\addtolength{\itemsep}{-2mm}
\footnotesize

\item Order 12: \{0: 1 2 3 4 5 6;   1: 0 2 3 4 7 8;   2: 0 1 5 7 9 10;   3: 0 1 5 8 9 11;   4: 0 1 6 7 9 11;   5: 0 2 3 6 10 11;   6: 0 4 5 8 9 10;   7: 1 2 4 8 10 11;   8: 1 3 6 7 9 10;   9: 2 3 4 6 8 11;  10: 2 5 6 7 8 11;  11: 3 4 5 7 9 10\}

\item Order 13: \{0: 1 2 3 4 5 6;   1: 0 2 3 4 5 7;   2: 0 1 3 4 8 9;   3: 0 1 2 6 10 11;   4: 0 1 2 8 10 12;   5: 0 1 6 7 11 12;   6: 0 3 5 7 8 9;   7: 1 5 6 9 10 11;   8: 2 4 6 9 10 12;   9: 2 6 7 8 11 12;  10: 3 4 7 8 11 12;  11: 3 5 7 9 10 12;  12: 4 5 8 9 10 11\}

\item Order 14: \{0: 1 2 3 4 5 6;   1: 0 2 3 4 5 6;   2: 0 1 3 4 5 7;   3: 0 1 2 6 8 9;   4: 0 1 2 7 10 11;   5: 0 1 2 8 9 10;   6: 0 1 3 8 11 12;   7: 2 4 8 10 11 13;   8: 3 5 6 7 12 13;   9: 3 5 10 11 12 13;  10: 4 5 7 9 12 13;  11: 4 6 7 9 12 13;  12: 6 8 9 10 11 13;  13: 7 8 9 10 11 12\}

\item Order 15: \{0: 1 2 3 4 5 6;   1: 0 2 3 4 5 6;   2: 0 1 3 4 5 7;   3: 0 1 2 4 5 8;   4: 0 1 2 3 7 9;   5: 0 1 2 3 10 11;   6: 0 1 10 12 13 14;   7: 2 4 9 12 13 14;   8: 3 9 10 11 12 13;   9: 4 7 8 11 12 14;  10: 5 6 8 11 12 13;  11: 5 8 9 10 13 14;  12: 6 7 8 9 10 14;  13: 6 7 8 10 11 14;  14: 6 7 9 11 12 13\}

\item Order 16: \{0: 1 2 3 4 5 6;   1: 0 2 3 4 5 6;   2: 0 1 3 4 5 6;   3: 0 1 2 4 5 7;   4: 0 1 2 3 7 8;   5: 0 1 2 3 9 10;   6: 0 1 2 7 9 11;   7: 3 4 6 8 12 13;   8: 4 7 11 12 14 15;   9: 5 6 10 11 14 15;  10: 5 9 11 12 13 14;  11: 6 8 9 10 13 15;  12: 7 8 10 13 14 15;  13: 7 10 11 12 14 15;  14: 8 9 10 12 13 15;  15: 8 9 11 12 13 14\}

\item Order 17: \{0: 1 2 3 4 5 6;   1: 0 2 3 4 5 6;   2: 0 1 3 4 5 6;   3: 0 1 2 4 5 6;   4: 0 1 2 3 7 8;   5: 0 1 2 3 7 9;   6: 0 1 2 3 10 11;   7: 4 5 10 12 13 14;   8: 4 9 11 12 15 16;   9: 5 8 13 14 15 16;  10: 6 7 11 13 15 16;  11: 6 8 10 14 15 16;  12: 7 8 13 14 15 16;  13: 7 9 10 12 14 15;  14: 7 9 11 12 13 16;  15: 8 9 10 11 12 13;  16: 8 9 10 11 12 14\}

\item Order 18: \{0: 1 2 3 4 5 6;   1: 0 2 3 4 5 6;   2: 0 1 3 4 5 6;   3: 0 1 2 4 5 6;   4: 0 1 2 3 5 7;   5: 0 1 2 3 4 7;   6: 0 1 2 3 8 9;   7: 4 5 8 9 10 11;   8: 6 7 10 12 13 14;   9: 6 7 12 13 15 16;  10: 7 8 11 14 15 17;  11: 7 10 12 13 14 16;  12: 8 9 11 14 16 17;  13: 8 9 11 15 16 17;  14: 8 10 11 12 15 17;  15: 9 10 13 14 16 17;  16: 9 11 12 13 15 17;  17: 10 12 13 14 15 16\}

\item Order 19: \{0: 1 2 3 4 5 6;   1: 0 2 3 4 5 6;   2: 0 1 3 4 5 6;   3: 0 1 2 4 5 6;   4: 0 1 2 3 5 6;   5: 0 1 2 3 4 7;   6: 0 1 2 3 4 8;   7: 5 8 9 10 11 12;   8: 6 7 9 10 13 14;   9: 7 8 10 11 12 15;  10: 7 8 9 13 16 17;  11: 7 9 12 14 15 16;  12: 7 9 11 15 17 18;  13: 8 10 14 16 17 18;  14: 8 11 13 15 16 18;  15: 9 11 12 14 17 18;  16: 10 11 13 14 17 18;  17: 10 12 13 15 16 18;  18: 12 13 14 15 16 17\}

\item Order 20: \{0: 1 2 3 4 5 6;   1: 0 2 3 4 5 6;   2: 0 1 3 4 5 6;   3: 0 1 2 4 7 8;   4: 0 1 2 3 7 9;   5: 0 1 2 6 7 8;   6: 0 1 2 5 7 9;   7: 3 4 5 6 8 10;   8: 3 5 7 9 10 11;   9: 4 6 8 10 11 12;  10: 7 8 9 12 13 14;  11: 8 9 12 13 14 15;  12: 9 10 11 16 17 18;  13: 10 11 15 16 17 19;  14: 10 11 15 16 18 19;  15: 11 13 14 17 18 19;  16: 12 13 14 17 18 19;  17: 12 13 15 16 18 19;  18: 12 14 15 16 17 19;  19: 13 14 15 16 17 18\}

\item Order 21: \{0: 1 2 3 4 5 6;   1: 0 2 3 4 5 6;   2: 0 1 3 4 5 6;   3: 0 1 2 4 7 8;   4: 0 1 2 3 7 9;   5: 0 1 2 6 7 8;   6: 0 1 2 5 7 9;   7: 3 4 5 6 8 9;   8: 3 5 7 9 10 11;   9: 4 6 7 8 10 11;  10: 8 9 12 13 14 15;  11: 8 9 16 17 18 19;  12: 10 13 14 15 16 17;  13: 10 12 14 16 18 20;  14: 10 12 13 18 19 20;  15: 10 12 16 17 19 20;  16: 11 12 13 15 17 19;  17: 11 12 15 16 18 20;  18: 11 13 14 17 19 20;  19: 11 14 15 16 18 20;  20: 13 14 15 17 18 19\}

\item Order 22: \{0: 1 2 3 4 5 6;   1: 0 2 3 4 5 6;   2: 0 1 3 4 5 6;   3: 0 1 2 4 7 8;   4: 0 1 2 3 7 9;   5: 0 1 2 6 7 8;   6: 0 1 2 5 7 9;   7: 3 4 5 6 8 9;   8: 3 5 7 9 10 11;   9: 4 6 7 8 10 11;  10: 8 9 11 12 13 14;  11: 8 9 10 12 15 16;  12: 10 11 13 17 18 19;  13: 10 12 15 16 20 21;  14: 10 15 17 18 20 21;  15: 11 13 14 16 17 20;  16: 11 13 15 17 19 21;  17: 12 14 15 16 18 19;  18: 12 14 17 19 20 21;  19: 12 16 17 18 20 21;  20: 13 14 15 18 19 21;  21: 13 14 16 18 19 20\}

\item Order 23: \{0: 1 2 3 4 5 6;   1: 0 2 3 4 5 6;   2: 0 1 3 4 5 6;   3: 0 1 2 4 7 8;   4: 0 1 2 3 7 9;   5: 0 1 2 6 7 8;   6: 0 1 2 5 7 9;   7: 3 4 5 6 8 9;   8: 3 5 7 9 10 11;   9: 4 6 7 8 10 11;  10: 8 9 11 12 13 14;  11: 8 9 10 12 13 15;  12: 10 11 13 14 16 17;  13: 10 11 12 15 16 18;  14: 10 12 15 17 18 19;  15: 11 13 14 17 20 21;  16: 12 13 17 20 21 22;  17: 12 14 15 16 19 22;  18: 13 14 19 20 21 22;  19: 14 17 18 20 21 22;  20: 15 16 18 19 21 22;  21: 15 16 18 19 20 22;  22: 16 17 18 19 20 21\}

\end{itemize}

\subsubsection*{Adjacency lists of 7-regular nut graphs}

\begin{itemize}
\addtolength{\itemsep}{-2mm}
\footnotesize

\item Order 12: \{0: 1 2 3 4 5 6 7;   1: 0 2 3 4 5 8 9;   2: 0 1 3 4 6 10 11;   3: 0 1 2 5 8 10 11;   4: 0 1 2 7 8 9 10;   5: 0 1 3 6 7 9 11;   6: 0 2 5 7 9 10 11;   7: 0 4 5 6 8 9 10;   8: 1 3 4 7 9 10 11;   9: 1 4 5 6 7 8 11;  10: 2 3 4 6 7 8 11;  11: 2 3 5 6 8 9 10\}

\item Order 14: \{0: 1 2 3 4 5 6 7;   1: 0 2 3 4 5 6 7;   2: 0 1 3 4 5 8 9;   3: 0 1 2 4 5 8 10;   4: 0 1 2 3 6 9 11;   5: 0 1 2 3 11 12 13;   6: 0 1 4 9 10 12 13;   7: 0 1 8 10 11 12 13;   8: 2 3 7 9 10 11 12;   9: 2 4 6 8 10 12 13;  10: 3 6 7 8 9 11 13;  11: 4 5 7 8 10 12 13;  12: 5 6 7 8 9 11 13;  13: 5 6 7 9 10 11 12\}

\item Order 16: \{0: 1 2 3 4 5 6 7;   1: 0 2 3 4 5 6 7;   2: 0 1 3 4 5 6 7;   3: 0 1 2 4 5 6 8;   4: 0 1 2 3 7 9 10;   5: 0 1 2 3 8 9 11;   6: 0 1 2 3 8 12 13;   7: 0 1 2 4 10 11 12;   8: 3 5 6 11 13 14 15;   9: 4 5 11 12 13 14 15;  10: 4 7 11 12 13 14 15;  11: 5 7 8 9 10 14 15;  12: 6 7 9 10 13 14 15;  13: 6 8 9 10 12 14 15;  14: 8 9 10 11 12 13 15;  15: 8 9 10 11 12 13 14\}

\item Order 18: \{0: 1 2 3 4 5 6 7;   1: 0 2 3 4 5 6 7;   2: 0 1 3 4 5 6 7;   3: 0 1 2 4 5 6 7;   4: 0 1 2 3 5 6 8;   5: 0 1 2 3 4 7 9;   6: 0 1 2 3 4 8 10;   7: 0 1 2 3 5 9 11;   8: 4 6 10 11 12 13 14;   9: 5 7 11 12 13 15 16;  10: 6 8 11 14 15 16 17;  11: 7 8 9 10 12 15 17;  12: 8 9 11 13 14 16 17;  13: 8 9 12 14 15 16 17;  14: 8 10 12 13 15 16 17;  15: 9 10 11 13 14 16 17;  16: 9 10 12 13 14 15 17;  17: 10 11 12 13 14 15 16\}

\item Order 20: \{0: 1 2 3 4 5 6 7;   1: 0 2 3 4 5 6 7;   2: 0 1 3 4 5 6 7;   3: 0 1 2 4 5 6 7;   4: 0 1 2 3 5 6 7;   5: 0 1 2 3 4 6 8;   6: 0 1 2 3 4 5 8;   7: 0 1 2 3 4 8 9;   8: 5 6 7 10 11 12 13;   9: 7 10 14 15 16 17 18;  10: 8 9 11 12 14 15 19;  11: 8 10 12 13 14 15 19;  12: 8 10 11 16 17 18 19;  13: 8 11 14 15 16 17 18;  14: 9 10 11 13 15 16 17;  15: 9 10 11 13 14 18 19;  16: 9 12 13 14 17 18 19;  17: 9 12 13 14 16 18 19;  18: 9 12 13 15 16 17 19;  19: 10 11 12 15 16 17 18\}

\item Order 22: \{0: 1 2 3 4 5 6 7;   1: 0 2 3 4 5 6 7;   2: 0 1 3 4 5 6 8;   3: 0 1 2 4 7 8 9;   4: 0 1 2 3 10 11 12;   5: 0 1 2 6 7 8 9;   6: 0 1 2 5 10 11 12;   7: 0 1 3 5 8 9 10;   8: 2 3 5 7 9 10 11;   9: 3 5 7 8 10 12 13;  10: 4 6 7 8 9 13 14;  11: 4 6 8 15 16 17 18;  12: 4 6 9 13 14 15 19;  13: 9 10 12 16 19 20 21;  14: 10 12 15 16 17 18 20;  15: 11 12 14 16 17 19 21;  16: 11 13 14 15 18 20 21;  17: 11 14 15 18 19 20 21;  18: 11 14 16 17 19 20 21;  19: 12 13 15 17 18 20 21;  20: 13 14 16 17 18 19 21;  21: 13 15 16 17 18 19 20\}

\item Order 24: \{0: 1 2 3 4 5 8 9;   1: 0 2 3 4 5 10 11;   2: 0 1 5 6 7 12 13;   3: 0 1 4 6 7 14 15;   4: 0 1 3 6 7 16 17;   5: 0 1 2 6 7 18 19;   6: 2 3 4 5 7 22 23;   7: 2 3 4 5 6 20 21;   8: 0 9 16 18 20 22 23;   9: 0 8 17 19 21 22 23;  10: 1 11 12 14 20 21 22;  11: 1 10 13 15 20 21 23;  12: 2 10 13 14 16 17 22;  13: 2 11 12 15 16 17 23;  14: 3 10 12 15 18 19 22;  15: 3 11 13 14 18 19 23;  16: 4 8 12 13 17 18 20;  17: 4 9 12 13 16 19 21;  18: 5 8 14 15 16 19 20;  19: 5 9 14 15 17 18 21;  20: 7 8 10 11 16 18 21;  21: 7 9 10 11 17 19 20;  22: 6 8 9 10 12 14 23;  23: 6 8 9 11 13 15 22\}

\end{itemize}

\subsubsection*{Adjacency lists of 8-regular nut graphs}


\begin{itemize}
\addtolength{\itemsep}{-2mm}
\footnotesize

 \item Order 12: \{0: 1 2 3 4 5 6 7 8;  1: 0 2 3 4 5 6 7 9;  2: 0 1 3 4 5 8 9 10;  3: 0 1 2 4 6 8 10 11;  4: 0 1 2 3 7 9 10 11;  5: 0 1 2 6 7 8 10 11;  6: 0 1 3 5 7 9 10 11;  7: 0 1 4 5 6 8 9 11;  8: 0 2 3 5 7 9 10 11;  9: 1 2 4 6 7 8 10 11; 10: 2 3 4 5 6 8 9 11; 11: 3 4 5 6 7 8 9 10\}

 \item Order 14: \{0: 1 2 3 4 5 6 7 8;  1: 0 2 3 4 5 6 7 8;  2: 0 1 3 4 5 6 9 10;  3: 0 1 2 4 5 11 12 13;  4: 0 1 2 3 7 8 9 11;  5: 0 1 2 3 9 10 11 12;  6: 0 1 2 7 10 11 12 13;  7: 0 1 4 6 9 10 12 13;  8: 0 1 4 9 10 11 12 13;  9: 2 4 5 7 8 10 11 13; 10: 2 5 6 7 8 9 12 13; 11: 3 4 5 6 8 9 12 13; 12: 3 5 6 7 8 10 11 13; 13: 3 6 7 8 9 10 11 12\}

 \item Order 15: \{0: 1 2 3 4 5 6 7 8;  1: 0 2 3 4 5 6 7 8;  2: 0 1 3 4 5 6 7 9;  3: 0 1 2 4 5 6 8 10;  4: 0 1 2 3 5 9 11 12;  5: 0 1 2 3 4 9 11 13;  6: 0 1 2 3 10 11 12 14;  7: 0 1 2 10 11 12 13 14;  8: 0 1 3 9 10 12 13 14;  9: 2 4 5 8 10 12 13 14; 10: 3 6 7 8 9 11 13 14; 11: 4 5 6 7 10 12 13 14; 12: 4 6 7 8 9 11 13 14; 13: 5 7 8 9 10 11 12 14; 14: 6 7 8 9 10 11 12 13\}

 \item Order 16: \{0: 1 2 3 4 5 6 7 8;  1: 0 2 3 4 5 6 7 8;  2: 0 1 3 4 5 6 7 8;  3: 0 1 2 4 5 6 7 9;  4: 0 1 2 3 5 9 10 11;  5: 0 1 2 3 4 12 13 14;  6: 0 1 2 3 10 12 13 15;  7: 0 1 2 3 10 12 14 15;  8: 0 1 2 9 11 13 14 15;  9: 3 4 8 10 11 12 13 15; 10: 4 6 7 9 11 12 13 14; 11: 4 8 9 10 12 13 14 15; 12: 5 6 7 9 10 11 14 15; 13: 5 6 8 9 10 11 14 15; 14: 5 7 8 10 11 12 13 15; 15: 6 7 8 9 11 12 13 14\}

 \item Order 17: \{0: 1 2 3 4 5 6 7 8;  1: 0 2 3 4 5 6 7 8;  2: 0 1 3 4 5 6 7 8;  3: 0 1 2 4 5 6 7 8;  4: 0 1 2 3 5 6 9 10;  5: 0 1 2 3 4 9 11 12;  6: 0 1 2 3 4 10 13 14;  7: 0 1 2 3 9 13 15 16;  8: 0 1 2 3 11 12 13 15;  9: 4 5 7 10 11 13 14 16; 10: 4 6 9 11 12 14 15 16; 11: 5 8 9 10 12 14 15 16; 12: 5 8 10 11 13 14 15 16; 13: 6 7 8 9 12 14 15 16; 14: 6 9 10 11 12 13 15 16; 15: 7 8 10 11 12 13 14 16; 16: 7 9 10 11 12 13 14 15\}

 \item Order 18: \{0: 1 2 3 4 5 6 7 8;  1: 0 2 3 4 5 6 7 8;  2: 0 1 3 4 5 6 7 8;  3: 0 1 2 4 5 6 7 8;  4: 0 1 2 3 5 6 7 9;  5: 0 1 2 3 4 6 9 10;  6: 0 1 2 3 4 5 11 12;  7: 0 1 2 3 4 9 13 14;  8: 0 1 2 3 9 13 15 16;  9: 4 5 7 8 10 11 15 17; 10: 5 9 11 12 14 15 16 17; 11: 6 9 10 12 13 14 16 17; 12: 6 10 11 13 14 15 16 17; 13: 7 8 11 12 14 15 16 17; 14: 7 10 11 12 13 15 16 17; 15: 8 9 10 12 13 14 16 17; 16: 8 10 11 12 13 14 15 17; 17: 9 10 11 12 13 14 15 16\}

 \item Order 19: \{0: 1 2 3 4 5 6 7 8;  1: 0 2 3 4 5 6 7 8;  2: 0 1 3 4 5 6 7 8;  3: 0 1 2 4 5 6 7 8;  4: 0 1 2 3 5 6 7 8;  5: 0 1 2 3 4 9 10 11;  6: 0 1 2 3 4 9 12 13;  7: 0 1 2 3 4 10 14 15;  8: 0 1 2 3 4 12 16 17;  9: 5 6 10 11 14 15 16 17; 10: 5 7 9 11 12 14 15 18; 11: 5 9 10 12 13 16 17 18; 12: 6 8 10 11 13 14 16 18; 13: 6 11 12 14 15 16 17 18; 14: 7 9 10 12 13 15 17 18; 15: 7 9 10 13 14 16 17 18; 16: 8 9 11 12 13 15 17 18; 17: 8 9 11 13 14 15 16 18; 18: 10 11 12 13 14 15 16 17\}

 \item Order 20: \{0: 1 2 3 4 5 6 7 8;  1: 0 2 3 4 5 6 7 8;  2: 0 1 3 4 5 6 7 8;  3: 0 1 2 4 5 6 7 8;  4: 0 1 2 3 5 6 7 8;  5: 0 1 2 3 4 6 7 8;  6: 0 1 2 3 4 5 9 10;  7: 0 1 2 3 4 5 9 11;  8: 0 1 2 3 4 5 12 13;  9: 6 7 10 11 12 14 15 16; 10: 6 9 12 13 14 15 17 18; 11: 7 9 13 14 15 16 17 19; 12: 8 9 10 14 16 17 18 19; 13: 8 10 11 15 16 17 18 19; 14: 9 10 11 12 16 17 18 19; 15: 9 10 11 13 16 17 18 19; 16: 9 11 12 13 14 15 18 19; 17: 10 11 12 13 14 15 18 19; 18: 10 12 13 14 15 16 17 19; 19: 11 12 13 14 15 16 17 18\}

 \item Order 21: \{0: 1 2 3 4 5 6 7 8;  1: 0 2 3 4 5 6 7 8;  2: 0 1 3 4 5 6 7 8;  3: 0 1 2 4 5 6 7 8;  4: 0 1 2 3 5 6 7 8;  5: 0 1 2 3 4 6 9 10;  6: 0 1 2 3 4 5 9 11;  7: 0 1 2 3 4 8 9 10;  8: 0 1 2 3 4 7 9 11;  9: 5 6 7 8 12 13 14 15; 10: 5 7 12 13 16 17 18 19; 11: 6 8 12 14 16 17 18 20; 12: 9 10 11 15 16 17 19 20; 13: 9 10 14 15 16 18 19 20; 14: 9 11 13 15 16 18 19 20; 15: 9 12 13 14 17 18 19 20; 16: 10 11 12 13 14 17 19 20; 17: 10 11 12 15 16 18 19 20; 18: 10 11 13 14 15 17 19 20; 19: 10 12 13 14 15 16 17 18; 20: 11 12 13 14 15 16 17 18\}

 \item Order 22: \{0: 1 2 3 4 5 6 7 8;  1: 0 2 3 4 5 6 7 8;  2: 0 1 3 4 5 6 7 8;  3: 0 1 2 4 5 6 7 8;  4: 0 1 2 3 5 6 7 8;  5: 0 1 2 3 4 6 9 10;  6: 0 1 2 3 4 5 9 11;  7: 0 1 2 3 4 8 9 10;  8: 0 1 2 3 4 7 9 11;  9: 5 6 7 8 10 11 12 13; 10: 5 7 9 11 12 13 14 15; 11: 6 8 9 10 12 14 16 17; 12: 9 10 11 13 15 18 19 20; 13: 9 10 12 14 16 18 19 21; 14: 10 11 13 15 17 18 20 21; 15: 10 12 14 16 17 19 20 21; 16: 11 13 15 17 18 19 20 21; 17: 11 14 15 16 18 19 20 21; 18: 12 13 14 16 17 19 20 21; 19: 12 13 15 16 17 18 20 21; 20: 12 14 15 16 17 18 19 21; 21: 13 14 15 16 17 18 19 20\}

 \item Order 23: \{0: 1 2 3 4 5 6 7 8;  1: 0 2 3 4 5 6 7 8;  2: 0 1 3 4 5 6 7 8;  3: 0 1 2 4 5 6 7 8;  4: 0 1 2 3 5 6 7 8;  5: 0 1 2 3 4 6 9 10;  6: 0 1 2 3 4 5 9 11;  7: 0 1 2 3 4 8 9 10;  8: 0 1 2 3 4 7 9 11;  9: 5 6 7 8 10 11 12 13; 10: 5 7 9 11 12 13 14 15; 11: 6 8 9 10 12 13 14 15; 12: 9 10 11 13 14 15 16 17; 13: 9 10 11 12 16 18 19 20; 14: 10 11 12 15 17 18 21 22; 15: 10 11 12 14 19 20 21 22; 16: 12 13 17 18 19 20 21 22; 17: 12 14 16 18 19 20 21 22; 18: 13 14 16 17 19 20 21 22; 19: 13 15 16 17 18 20 21 22; 20: 13 15 16 17 18 19 21 22; 21: 14 15 16 17 18 19 20 22; 22: 14 15 16 17 18 19 20 21\}

 \item Order 24: \{0: 1 2 3 4 5 6 7 8;  1: 0 2 3 4 5 6 7 8;  2: 0 1 3 4 5 6 7 8;  3: 0 1 2 4 5 6 7 8;  4: 0 1 2 3 5 6 7 8;  5: 0 1 2 3 4 6 9 10;  6: 0 1 2 3 4 5 9 11;  7: 0 1 2 3 4 8 9 10;  8: 0 1 2 3 4 7 9 11;  9: 5 6 7 8 10 11 12 13; 10: 5 7 9 11 12 13 14 15; 11: 6 8 9 10 12 13 14 15; 12: 9 10 11 13 14 15 16 17; 13: 9 10 11 12 14 16 17 18; 14: 10 11 12 13 15 19 20 21; 15: 10 11 12 14 16 19 22 23; 16: 12 13 15 18 20 21 22 23; 17: 12 13 18 19 20 21 22 23; 18: 13 16 17 19 20 21 22 23; 19: 14 15 17 18 20 21 22 23; 20: 14 16 17 18 19 21 22 23; 21: 14 16 17 18 19 20 22 23; 22: 15 16 17 18 19 20 21 23; 23: 15 16 17 18 19 20 21 22\}

 \item Order 25: \{0: 1 2 3 4 5 6 7 8;  1: 0 2 3 4 5 6 7 8;  2: 0 1 3 4 5 6 7 8;  3: 0 1 2 4 5 6 7 8;  4: 0 1 2 3 5 6 7 8;  5: 0 1 2 3 4 6 9 10;  6: 0 1 2 3 4 5 9 11;  7: 0 1 2 3 4 8 9 10;  8: 0 1 2 3 4 7 9 11;  9: 5 6 7 8 10 11 12 13; 10: 5 7 9 11 12 13 14 15; 11: 6 8 9 10 12 13 14 15; 12: 9 10 11 13 14 15 16 17; 13: 9 10 11 12 14 15 18 19; 14: 10 11 12 13 16 17 20 21; 15: 10 11 12 13 16 20 22 23; 16: 12 14 15 18 19 21 22 24; 17: 12 14 18 20 21 22 23 24; 18: 13 16 17 19 21 22 23 24; 19: 13 16 18 20 21 22 23 24; 20: 14 15 17 19 21 22 23 24; 21: 14 16 17 18 19 20 23 24; 22: 15 16 17 18 19 20 23 24; 23: 15 17 18 19 20 21 22 24; 24: 16 17 18 19 20 21 22 23\}

 \item Order 26: \{0: 1 2 3 4 5 6 7 8;  1: 0 2 3 4 5 6 7 8;  2: 0 1 3 4 5 6 7 8;  3: 0 1 2 4 5 6 9 10;  4: 0 1 2 3 5 9 11 12;  5: 0 1 2 3 4 11 12 13;  6: 0 1 2 3 11 14 15 16;  7: 0 1 2 8 9 10 11 12;  8: 0 1 2 7 9 10 11 12;  9: 3 4 7 8 10 11 12 13; 10: 3 7 8 9 11 12 13 14; 11: 4 5 6 7 8 9 10 12; 12: 4 5 7 8 9 10 11 13; 13: 5 9 10 12 14 15 16 17; 14: 6 10 13 15 16 17 18 19; 15: 6 13 14 16 17 18 19 20; 16: 6 13 14 15 18 21 22 23; 17: 13 14 15 19 21 22 24 25; 18: 14 15 16 20 21 23 24 25; 19: 14 15 17 20 22 23 24 25; 20: 15 18 19 21 22 23 24 25; 21: 16 17 18 20 22 23 24 25; 22: 16 17 19 20 21 23 24 25; 23: 16 18 19 20 21 22 24 25; 24: 17 18 19 20 21 22 23 25; 25: 17 18 19 20 21 22 23 24\}

 \item Order 27: \{0: 1 2 3 4 5 6 7 8;  1: 0 2 3 4 5 6 7 8;  2: 0 1 3 4 5 6 7 8;  3: 0 1 2 4 5 6 7 8;  4: 0 1 2 3 5 6 7 8;  5: 0 1 2 3 4 6 9 10;  6: 0 1 2 3 4 5 9 11;  7: 0 1 2 3 4 8 9 10;  8: 0 1 2 3 4 7 9 11;  9: 5 6 7 8 10 11 12 13; 10: 5 7 9 11 12 13 14 15; 11: 6 8 9 10 12 13 14 15; 12: 9 10 11 13 14 15 16 17; 13: 9 10 11 12 14 15 16 17; 14: 10 11 12 13 15 16 17 18; 15: 10 11 12 13 14 16 18 19; 16: 12 13 14 15 17 20 21 22; 17: 12 13 14 16 18 19 20 23; 18: 14 15 17 19 23 24 25 26; 19: 15 17 18 21 22 24 25 26; 20: 16 17 21 22 23 24 25 26; 21: 16 19 20 22 23 24 25 26; 22: 16 19 20 21 23 24 25 26; 23: 17 18 20 21 22 24 25 26; 24: 18 19 20 21 22 23 25 26; 25: 18 19 20 21 22 23 24 26; 26: 18 19 20 21 22 23 24 25\}


 \item Order 29: \{0: 1 2 3 4 5 6 7 8;  1: 0 2 3 4 5 6 7 8;  2: 0 1 3 4 5 6 7 8;  3: 0 1 2 4 5 6 7 8;  4: 0 1 2 3 5 6 9 10;  5: 0 1 2 3 4 7 9 10;  6: 0 1 2 3 4 9 10 11;  7: 0 1 2 3 5 9 12 13;  8: 0 1 2 3 9 11 12 14;  9: 4 5 6 7 8 10 11 12; 10: 4 5 6 9 11 12 13 14; 11: 6 8 9 10 12 13 14 15; 12: 7 8 9 10 11 13 14 15; 13: 7 10 11 12 14 15 16 17; 14: 8 10 11 12 13 15 16 17; 15: 11 12 13 14 16 17 18 19; 16: 13 14 15 17 18 19 20 21; 17: 13 14 15 16 18 19 20 22; 18: 15 16 17 21 22 23 24 25; 19: 15 16 17 23 24 26 27 28; 20: 16 17 21 23 25 26 27 28; 21: 16 18 20 22 23 24 25 26; 22: 17 18 21 23 24 25 27 28; 23: 18 19 20 21 22 26 27 28; 24: 18 19 21 22 25 26 27 28; 25: 18 20 21 22 24 26 27 28; 26: 19 20 21 23 24 25 27 28; 27: 19 20 22 23 24 25 26 28; 28: 19 20 22 23 24 25 26 27\}

\end{itemize}

\subsubsection*{Adjacency lists of 9-regular nut graphs}

\begin{itemize}
\addtolength{\itemsep}{-2mm}
\footnotesize

 \item Order 16: \{0: 1 2 3 4 5 6 7 8 9;  1: 0 2 3 4 5 6 7 8 9;  2: 0 1 3 4 5 6 7 8 9;  3: 0 1 2 4 5 6 7 10 11;  4: 0 1 2 3 5 8 10 12 13;  5: 0 1 2 3 4 11 12 13 14;  6: 0 1 2 3 9 11 12 14 15;  7: 0 1 2 3 10 11 12 13 15;  8: 0 1 2 4 9 10 13 14 15;  9: 0 1 2 6 8 11 13 14 15; 10: 3 4 7 8 11 12 13 14 15; 11: 3 5 6 7 9 10 12 14 15; 12: 4 5 6 7 10 11 13 14 15; 13: 4 5 7 8 9 10 12 14 15; 14: 5 6 8 9 10 11 12 13 15; 15: 6 7 8 9 10 11 12 13 14\}

 \item Order 18: \{0: 1 2 3 4 5 6 7 8 9;  1: 0 2 3 4 5 6 7 8 9;  2: 0 1 3 4 5 6 7 8 9;  3: 0 1 2 4 5 6 7 8 9;  4: 0 1 2 3 5 6 7 8 10;  5: 0 1 2 3 4 6 11 12 13;  6: 0 1 2 3 4 5 14 15 16;  7: 0 1 2 3 4 10 11 14 17;  8: 0 1 2 3 4 11 12 15 16;  9: 0 1 2 3 10 11 12 13 17; 10: 4 7 9 12 13 14 15 16 17; 11: 5 7 8 9 13 14 15 16 17; 12: 5 8 9 10 13 14 15 16 17; 13: 5 9 10 11 12 14 15 16 17; 14: 6 7 10 11 12 13 15 16 17; 15: 6 8 10 11 12 13 14 16 17; 16: 6 8 10 11 12 13 14 15 17; 17: 7 9 10 11 12 13 14 15 16\}

 \item Order 20: \{0: 1 2 3 4 5 6 7 8 9;  1: 0 2 3 4 5 6 7 8 9;  2: 0 1 3 4 5 6 7 8 9;  3: 0 1 2 4 5 6 7 8 9;  4: 0 1 2 3 5 6 7 8 9;  5: 0 1 2 3 4 6 7 8 10;  6: 0 1 2 3 4 5 10 11 12;  7: 0 1 2 3 4 5 11 13 14;  8: 0 1 2 3 4 5 12 13 15;  9: 0 1 2 3 4 11 14 15 16; 10: 5 6 12 14 15 16 17 18 19; 11: 6 7 9 12 13 15 17 18 19; 12: 6 8 10 11 13 16 17 18 19; 13: 7 8 11 12 14 16 17 18 19; 14: 7 9 10 13 15 16 17 18 19; 15: 8 9 10 11 14 16 17 18 19; 16: 9 10 12 13 14 15 17 18 19; 17: 10 11 12 13 14 15 16 18 19; 18: 10 11 12 13 14 15 16 17 19; 19: 10 11 12 13 14 15 16 17 18\}

 \item Order 22: \{0: 1 2 3 4 5 6 7 8 9;  1: 0 2 3 4 5 6 7 8 9;  2: 0 1 3 4 5 6 7 8 9;  3: 0 1 2 4 5 6 7 8 9;  4: 0 1 2 3 5 6 7 8 9;  5: 0 1 2 3 4 6 7 8 9;  6: 0 1 2 3 4 5 7 8 10;  7: 0 1 2 3 4 5 6 8 10;  8: 0 1 2 3 4 5 6 7 11;  9: 0 1 2 3 4 5 11 12 13; 10: 6 7 12 13 14 15 16 17 18; 11: 8 9 12 14 15 16 19 20 21; 12: 9 10 11 13 14 15 17 19 20; 13: 9 10 12 14 16 17 18 19 21; 14: 10 11 12 13 15 16 18 20 21; 15: 10 11 12 14 17 18 19 20 21; 16: 10 11 13 14 17 18 19 20 21; 17: 10 12 13 15 16 18 19 20 21; 18: 10 13 14 15 16 17 19 20 21; 19: 11 12 13 15 16 17 18 20 21; 20: 11 12 14 15 16 17 18 19 21; 21: 11 13 14 15 16 17 18 19 20\}

 \item Order 24: \{0: 1 2 3 4 5 6 7 8 9;  1: 0 2 3 4 5 6 7 8 9;  2: 0 1 3 4 5 6 7 8 9;  3: 0 1 2 4 5 6 10 11 12;  4: 0 1 2 3 5 6 10 11 13;  5: 0 1 2 3 4 14 15 16 17;  6: 0 1 2 3 4 18 19 20 21;  7: 0 1 2 8 10 14 15 18 19;  8: 0 1 2 7 10 14 15 18 19;  9: 0 1 2 10 11 12 13 14 15; 10: 3 4 7 8 9 11 12 13 14; 11: 3 4 9 10 12 13 14 15 16; 12: 3 9 10 11 13 14 15 16 22; 13: 4 9 10 11 12 17 20 22 23; 14: 5 7 8 9 10 11 12 21 22; 15: 5 7 8 9 11 12 18 20 23; 16: 5 11 12 17 19 20 21 22 23; 17: 5 13 16 18 19 20 21 22 23; 18: 6 7 8 15 17 20 21 22 23; 19: 6 7 8 16 17 20 21 22 23; 20: 6 13 15 16 17 18 19 21 23; 21: 6 14 16 17 18 19 20 22 23; 22: 12 13 14 16 17 18 19 21 23; 23: 13 15 16 17 18 19 20 21 22\}

 \item Order 26: \{0: 1 2 3 4 5 6 7 8 9;  1: 0 2 3 4 5 6 7 8 9;  2: 0 1 3 4 5 6 7 8 9;  3: 0 1 2 4 5 6 7 8 9;  4: 0 1 2 3 5 6 7 8 10;  5: 0 1 2 3 4 6 7 11 12;  6: 0 1 2 3 4 5 11 13 14;  7: 0 1 2 3 4 5 15 16 17;  8: 0 1 2 3 4 9 13 15 16;  9: 0 1 2 3 8 10 11 12 13; 10: 4 9 11 12 13 14 15 16 17; 11: 5 6 9 10 12 13 14 15 16; 12: 5 9 10 11 13 14 15 16 17; 13: 6 8 9 10 11 12 18 19 20; 14: 6 10 11 12 15 18 19 21 22; 15: 7 8 10 11 12 14 18 21 23; 16: 7 8 10 11 12 17 20 24 25; 17: 7 10 12 16 20 22 23 24 25; 18: 13 14 15 19 21 22 23 24 25; 19: 13 14 18 20 21 22 23 24 25; 20: 13 16 17 19 21 22 23 24 25; 21: 14 15 18 19 20 22 23 24 25; 22: 14 17 18 19 20 21 23 24 25; 23: 15 17 18 19 20 21 22 24 25; 24: 16 17 18 19 20 21 22 23 25; 25: 16 17 18 19 20 21 22 23 24\}

 \item Order 28: \{0: 1 2 3 4 5 6 7 8 9;  1: 0 2 3 4 5 6 7 8 9;  2: 0 1 3 4 5 6 7 8 9;  3: 0 1 2 4 5 6 10 11 12;  4: 0 1 2 3 5 6 10 11 13;  5: 0 1 2 3 4 10 12 13 14;  6: 0 1 2 3 4 12 15 16 17;  7: 0 1 2 8 9 10 12 15 18;  8: 0 1 2 7 9 10 12 15 18;  9: 0 1 2 7 8 10 12 15 18; 10: 3 4 5 7 8 9 11 12 13; 11: 3 4 10 12 13 14 15 16 17; 12: 3 5 6 7 8 9 10 11 13; 13: 4 5 10 11 12 14 15 16 17; 14: 5 11 13 15 16 17 18 19 20; 15: 6 7 8 9 11 13 14 16 19; 16: 6 11 13 14 15 19 21 22 23; 17: 6 11 13 14 19 20 24 25 26; 18: 7 8 9 14 21 22 23 24 27; 19: 14 15 16 17 20 21 25 26 27; 20: 14 17 19 22 23 24 25 26 27; 21: 16 18 19 22 23 24 25 26 27; 22: 16 18 20 21 23 24 25 26 27; 23: 16 18 20 21 22 24 25 26 27; 24: 17 18 20 21 22 23 25 26 27; 25: 17 19 20 21 22 23 24 26 27; 26: 17 19 20 21 22 23 24 25 27; 27: 18 19 20 21 22 23 24 25 26\}

 \item Order 30: \{0: 1 2 3 4 5 6 7 8 9;  1: 0 2 3 4 5 6 7 8 9;  2: 0 1 3 4 5 6 7 8 9;  3: 0 1 2 4 5 6 7 8 9;  4: 0 1 2 3 5 6 7 8 9;  5: 0 1 2 3 4 6 7 8 9;  6: 0 1 2 3 4 5 7 10 11;  7: 0 1 2 3 4 5 6 12 13;  8: 0 1 2 3 4 5 10 11 12;  9: 0 1 2 3 4 5 10 14 15; 10: 6 8 9 12 13 16 17 18 19; 11: 6 8 13 20 21 22 23 24 25; 12: 7 8 10 13 14 15 16 17 18; 13: 7 10 11 12 14 15 16 17 18; 14: 9 12 13 15 16 17 18 19 20; 15: 9 12 13 14 16 17 18 19 20; 16: 10 12 13 14 15 17 18 19 20; 17: 10 12 13 14 15 16 18 19 20; 18: 10 12 13 14 15 16 17 21 26; 19: 10 14 15 16 17 22 23 24 26; 20: 11 14 15 16 17 22 27 28 29; 21: 11 18 23 24 25 26 27 28 29; 22: 11 19 20 23 25 26 27 28 29; 23: 11 19 21 22 24 25 27 28 29; 24: 11 19 21 23 25 26 27 28 29; 25: 11 21 22 23 24 26 27 28 29; 26: 18 19 21 22 24 25 27 28 29; 27: 20 21 22 23 24 25 26 28 29; 28: 20 21 22 23 24 25 26 27 29; 29: 20 21 22 23 24 25 26 27 28\}

 \item Order 32: \{0: 1 2 3 4 5 6 7 8 9;  1: 0 2 3 4 5 6 7 8 9;  2: 0 1 3 4 5 6 7 8 9;  3: 0 1 2 4 5 6 7 8 9;  4: 0 1 2 3 5 6 7 8 9;  5: 0 1 2 3 4 6 7 8 9;  6: 0 1 2 3 4 5 7 10 11;  7: 0 1 2 3 4 5 6 12 13;  8: 0 1 2 3 4 5 10 11 12;  9: 0 1 2 3 4 5 10 14 15; 10: 6 8 9 11 14 16 17 18 19; 11: 6 8 10 12 13 15 16 20 21; 12: 7 8 11 13 14 15 16 17 18; 13: 7 11 12 14 15 16 17 18 19; 14: 9 10 12 13 15 16 17 18 19; 15: 9 11 12 13 14 16 17 18 19; 16: 10 11 12 13 14 15 17 18 19; 17: 10 12 13 14 15 16 18 19 20; 18: 10 12 13 14 15 16 17 19 22; 19: 10 13 14 15 16 17 18 23 24; 20: 11 17 21 22 23 25 26 27 28; 21: 11 20 22 23 25 26 27 29 30; 22: 18 20 21 23 24 25 29 30 31; 23: 19 20 21 22 24 25 28 29 31; 24: 19 22 23 26 27 28 29 30 31; 25: 20 21 22 23 26 27 28 30 31; 26: 20 21 24 25 27 28 29 30 31; 27: 20 21 24 25 26 28 29 30 31; 28: 20 23 24 25 26 27 29 30 31; 29: 21 22 23 24 26 27 28 30 31; 30: 21 22 24 25 26 27 28 29 31; 31: 22 23 24 25 26 27 28 29 30\}

\end{itemize}

\subsubsection*{Adjacency lists of 10-regular nut graphs}

\begin{itemize}
\addtolength{\itemsep}{-2mm}
\footnotesize

 \item Order 15: \{0: 1 2 3 4 5 6 7 8 9 10;  1: 0 2 3 4 5 6 7 8 9 11;  2: 0 1 3 4 5 6 7 8 10 12;  3: 0 1 2 4 5 6 7 9 11 13;  4: 0 1 2 3 5 6 7 10 13 14;  5: 0 1 2 3 4 6 9 11 12 14;  6: 0 1 2 3 4 5 11 12 13 14;  7: 0 1 2 3 4 8 10 12 13 14;  8: 0 1 2 7 9 10 11 12 13 14;  9: 0 1 3 5 8 10 11 12 13 14; 10: 0 2 4 7 8 9 11 12 13 14; 11: 1 3 5 6 8 9 10 12 13 14; 12: 2 5 6 7 8 9 10 11 13 14; 13: 3 4 6 7 8 9 10 11 12 14; 14: 4 5 6 7 8 9 10 11 12 13\}

 \item Order 16: \{0: 1 2 3 4 5 6 7 8 9 10;  1: 0 2 3 4 5 6 7 8 9 10;  2: 0 1 3 4 5 6 7 8 9 11;  3: 0 1 2 4 5 6 7 8 10 11;  4: 0 1 2 3 5 6 11 12 13 14;  5: 0 1 2 3 4 7 12 13 14 15;  6: 0 1 2 3 4 8 12 13 14 15;  7: 0 1 2 3 5 9 10 11 12 15;  8: 0 1 2 3 6 9 12 13 14 15;  9: 0 1 2 7 8 10 11 13 14 15; 10: 0 1 3 7 9 11 12 13 14 15; 11: 2 3 4 7 9 10 12 13 14 15; 12: 4 5 6 7 8 10 11 13 14 15; 13: 4 5 6 8 9 10 11 12 14 15; 14: 4 5 6 8 9 10 11 12 13 15; 15: 5 6 7 8 9 10 11 12 13 14\}

 \item Order 17: \{0: 1 2 3 4 5 6 7 8 9 10;  1: 0 2 3 4 5 6 7 8 9 10;  2: 0 1 3 4 5 6 7 8 9 10;  3: 0 1 2 4 5 6 7 8 9 11;  4: 0 1 2 3 5 6 7 10 12 13;  5: 0 1 2 3 4 8 12 13 14 15;  6: 0 1 2 3 4 10 11 12 14 16;  7: 0 1 2 3 4 12 13 14 15 16;  8: 0 1 2 3 5 11 13 14 15 16;  9: 0 1 2 3 11 12 13 14 15 16; 10: 0 1 2 4 6 11 12 14 15 16; 11: 3 6 8 9 10 12 13 14 15 16; 12: 4 5 6 7 9 10 11 13 15 16; 13: 4 5 7 8 9 11 12 14 15 16; 14: 5 6 7 8 9 10 11 13 15 16; 15: 5 7 8 9 10 11 12 13 14 16; 16: 6 7 8 9 10 11 12 13 14 15\}

 \item Order 18: \{0: 1 2 3 4 5 6 7 8 9 10;  1: 0 2 3 4 5 6 7 8 9 10;  2: 0 1 3 4 5 6 7 8 9 10;  3: 0 1 2 4 5 6 7 8 9 10;  4: 0 1 2 3 5 6 7 8 11 12;  5: 0 1 2 3 4 6 9 13 14 15;  6: 0 1 2 3 4 5 10 13 16 17;  7: 0 1 2 3 4 11 12 13 14 15;  8: 0 1 2 3 4 13 14 15 16 17;  9: 0 1 2 3 5 11 12 14 16 17; 10: 0 1 2 3 6 11 12 15 16 17; 11: 4 7 9 10 12 13 14 15 16 17; 12: 4 7 9 10 11 13 14 15 16 17; 13: 5 6 7 8 11 12 14 15 16 17; 14: 5 7 8 9 11 12 13 15 16 17; 15: 5 7 8 10 11 12 13 14 16 17; 16: 6 8 9 10 11 12 13 14 15 17; 17: 6 8 9 10 11 12 13 14 15 16\}

 \item Order 19: \{0: 1 2 3 4 5 6 7 8 9 10;  1: 0 2 3 4 5 6 7 8 9 10;  2: 0 1 3 4 5 6 7 8 9 10;  3: 0 1 2 4 5 6 7 8 9 10;  4: 0 1 2 3 5 6 7 8 9 11;  5: 0 1 2 3 4 6 7 10 11 12;  6: 0 1 2 3 4 5 10 13 14 15;  7: 0 1 2 3 4 5 12 13 16 17;  8: 0 1 2 3 4 11 14 16 17 18;  9: 0 1 2 3 4 13 15 16 17 18; 10: 0 1 2 3 5 6 12 14 15 18; 11: 4 5 8 12 13 14 15 16 17 18; 12: 5 7 10 11 13 14 15 16 17 18; 13: 6 7 9 11 12 14 15 16 17 18; 14: 6 8 10 11 12 13 15 16 17 18; 15: 6 9 10 11 12 13 14 16 17 18; 16: 7 8 9 11 12 13 14 15 17 18; 17: 7 8 9 11 12 13 14 15 16 18; 18: 8 9 10 11 12 13 14 15 16 17\}

 \item Order 20: \{0: 1 2 3 4 5 6 7 8 9 10;  1: 0 2 3 4 5 6 7 8 9 10;  2: 0 1 3 4 5 6 7 8 9 10;  3: 0 1 2 4 5 6 7 8 9 10;  4: 0 1 2 3 5 6 7 8 9 10;  5: 0 1 2 3 4 6 7 8 11 12;  6: 0 1 2 3 4 5 7 9 11 13;  7: 0 1 2 3 4 5 6 14 15 16;  8: 0 1 2 3 4 5 11 13 17 18;  9: 0 1 2 3 4 6 12 14 15 19; 10: 0 1 2 3 4 13 16 17 18 19; 11: 5 6 8 12 14 15 16 17 18 19; 12: 5 9 11 13 14 15 16 17 18 19; 13: 6 8 10 12 14 15 16 17 18 19; 14: 7 9 11 12 13 15 16 17 18 19; 15: 7 9 11 12 13 14 16 17 18 19; 16: 7 10 11 12 13 14 15 17 18 19; 17: 8 10 11 12 13 14 15 16 18 19; 18: 8 10 11 12 13 14 15 16 17 19; 19: 9 10 11 12 13 14 15 16 17 18\}

 \item Order 21: \{0: 1 2 3 4 5 6 7 8 9 10;  1: 0 2 3 4 5 6 7 8 9 10;  2: 0 1 3 4 5 6 7 8 9 10;  3: 0 1 2 4 5 6 7 8 9 10;  4: 0 1 2 3 5 6 7 8 9 10;  5: 0 1 2 3 4 6 7 8 9 11;  6: 0 1 2 3 4 5 7 8 10 12;  7: 0 1 2 3 4 5 6 11 13 14;  8: 0 1 2 3 4 5 6 15 16 17;  9: 0 1 2 3 4 5 13 14 18 19; 10: 0 1 2 3 4 6 12 13 15 20; 11: 5 7 12 14 15 16 17 18 19 20; 12: 6 10 11 13 15 16 17 18 19 20; 13: 7 9 10 12 14 16 17 18 19 20; 14: 7 9 11 13 15 16 17 18 19 20; 15: 8 10 11 12 14 16 17 18 19 20; 16: 8 11 12 13 14 15 17 18 19 20; 17: 8 11 12 13 14 15 16 18 19 20; 18: 9 11 12 13 14 15 16 17 19 20; 19: 9 11 12 13 14 15 16 17 18 20; 20: 10 11 12 13 14 15 16 17 18 19\}

 \item Order 22: \{0: 1 2 3 4 5 6 7 8 9 10;  1: 0 2 3 4 5 6 7 8 9 10;  2: 0 1 3 4 5 6 7 8 9 10;  3: 0 1 2 4 5 6 7 8 9 10;  4: 0 1 2 3 5 6 7 8 9 10;  5: 0 1 2 3 4 6 7 8 9 10;  6: 0 1 2 3 4 5 7 8 9 11;  7: 0 1 2 3 4 5 6 8 11 12;  8: 0 1 2 3 4 5 6 7 11 13;  9: 0 1 2 3 4 5 6 14 15 16; 10: 0 1 2 3 4 5 11 12 14 17; 11: 6 7 8 10 13 15 18 19 20 21; 12: 7 10 13 14 16 17 18 19 20 21; 13: 8 11 12 15 16 17 18 19 20 21; 14: 9 10 12 15 16 17 18 19 20 21; 15: 9 11 13 14 16 17 18 19 20 21; 16: 9 12 13 14 15 17 18 19 20 21; 17: 10 12 13 14 15 16 18 19 20 21; 18: 11 12 13 14 15 16 17 19 20 21; 19: 11 12 13 14 15 16 17 18 20 21; 20: 11 12 13 14 15 16 17 18 19 21; 21: 11 12 13 14 15 16 17 18 19 20\}

 \item Order 23: \{0: 1 2 3 4 5 6 7 8 9 10;  1: 0 2 3 4 5 6 7 8 9 10;  2: 0 1 3 4 5 6 7 8 9 10;  3: 0 1 2 4 5 6 7 8 9 10;  4: 0 1 2 3 5 6 7 8 9 10;  5: 0 1 2 3 4 6 7 11 12 13;  6: 0 1 2 3 4 5 8 11 12 14;  7: 0 1 2 3 4 5 11 15 16 17;  8: 0 1 2 3 4 6 13 18 19 20;  9: 0 1 2 3 4 10 11 12 13 14; 10: 0 1 2 3 4 9 11 12 13 15; 11: 5 6 7 9 10 15 16 18 19 21; 12: 5 6 9 10 13 16 17 18 19 22; 13: 5 8 9 10 12 14 18 20 21 22; 14: 6 9 13 15 16 17 19 20 21 22; 15: 7 10 11 14 17 18 19 20 21 22; 16: 7 11 12 14 17 18 19 20 21 22; 17: 7 12 14 15 16 18 19 20 21 22; 18: 8 11 12 13 15 16 17 20 21 22; 19: 8 11 12 14 15 16 17 20 21 22; 20: 8 13 14 15 16 17 18 19 21 22; 21: 11 13 14 15 16 17 18 19 20 22; 22: 12 13 14 15 16 17 18 19 20 21\}

 \item Order 24: \{0: 1 2 3 4 5 6 7 8 9 10;  1: 0 2 3 4 5 6 7 8 9 10;  2: 0 1 3 4 5 6 7 8 9 10;  3: 0 1 2 4 5 6 7 8 9 10;  4: 0 1 2 3 5 6 7 8 9 10;  5: 0 1 2 3 4 6 7 11 12 13;  6: 0 1 2 3 4 5 8 11 12 14;  7: 0 1 2 3 4 5 11 15 16 17;  8: 0 1 2 3 4 6 12 18 19 20;  9: 0 1 2 3 4 10 11 12 13 14; 10: 0 1 2 3 4 9 11 12 13 14; 11: 5 6 7 9 10 12 13 14 15 18; 12: 5 6 8 9 10 11 15 16 19 21; 13: 5 9 10 11 17 18 20 21 22 23; 14: 6 9 10 11 15 17 19 20 22 23; 15: 7 11 12 14 16 19 20 21 22 23; 16: 7 12 15 17 18 19 20 21 22 23; 17: 7 13 14 16 18 19 20 21 22 23; 18: 8 11 13 16 17 19 20 21 22 23; 19: 8 12 14 15 16 17 18 21 22 23; 20: 8 13 14 15 16 17 18 21 22 23; 21: 12 13 15 16 17 18 19 20 22 23; 22: 13 14 15 16 17 18 19 20 21 23; 23: 13 14 15 16 17 18 19 20 21 22\}

 \item Order 25: \{0: 1 2 3 4 5 6 7 8 9 10;  1: 0 2 3 4 5 6 7 8 9 10;  2: 0 1 3 4 5 6 7 8 9 10;  3: 0 1 2 4 5 6 7 8 9 10;  4: 0 1 2 3 5 6 7 8 9 10;  5: 0 1 2 3 4 6 7 11 12 13;  6: 0 1 2 3 4 5 8 11 12 14;  7: 0 1 2 3 4 5 11 15 16 17;  8: 0 1 2 3 4 6 12 15 16 18;  9: 0 1 2 3 4 10 11 12 13 14; 10: 0 1 2 3 4 9 11 12 13 14; 11: 5 6 7 9 10 12 13 14 15 16; 12: 5 6 8 9 10 11 13 15 17 19; 13: 5 9 10 11 12 17 19 20 21 22; 14: 6 9 10 11 19 20 21 22 23 24; 15: 7 8 11 12 17 18 20 21 23 24; 16: 7 8 11 17 18 19 20 22 23 24; 17: 7 12 13 15 16 18 21 22 23 24; 18: 8 15 16 17 19 20 21 22 23 24; 19: 12 13 14 16 18 20 21 22 23 24; 20: 13 14 15 16 18 19 21 22 23 24; 21: 13 14 15 17 18 19 20 22 23 24; 22: 13 14 16 17 18 19 20 21 23 24; 23: 14 15 16 17 18 19 20 21 22 24; 24: 14 15 16 17 18 19 20 21 22 23\}

 \item Order 26: \{0: 1 2 3 4 5 6 7 8 9 10;  1: 0 2 3 4 5 6 7 8 9 10;  2: 0 1 3 4 5 6 7 8 9 10;  3: 0 1 2 4 5 6 7 8 9 10;  4: 0 1 2 3 5 6 7 8 9 10;  5: 0 1 2 3 4 6 7 11 12 13;  6: 0 1 2 3 4 5 8 11 12 14;  7: 0 1 2 3 4 5 11 15 16 17;  8: 0 1 2 3 4 6 12 15 16 18;  9: 0 1 2 3 4 10 11 12 13 14; 10: 0 1 2 3 4 9 11 12 13 14; 11: 5 6 7 9 10 12 13 14 15 16; 12: 5 6 8 9 10 11 13 14 15 18; 13: 5 9 10 11 12 15 16 19 20 21; 14: 6 9 10 11 12 17 19 20 21 22; 15: 7 8 11 12 13 16 22 23 24 25; 16: 7 8 11 13 15 17 19 23 24 25; 17: 7 14 16 18 20 21 22 23 24 25; 18: 8 12 17 19 20 21 22 23 24 25; 19: 13 14 16 18 20 21 22 23 24 25; 20: 13 14 17 18 19 21 22 23 24 25; 21: 13 14 17 18 19 20 22 23 24 25; 22: 14 15 17 18 19 20 21 23 24 25; 23: 15 16 17 18 19 20 21 22 24 25; 24: 15 16 17 18 19 20 21 22 23 25; 25: 15 16 17 18 19 20 21 22 23 24\}

 \item Order 27: \{0: 1 2 3 4 5 6 7 8 9 10;  1: 0 2 3 4 5 6 7 8 9 10;  2: 0 1 3 4 5 6 7 8 9 10;  3: 0 1 2 4 5 6 7 8 9 10;  4: 0 1 2 3 5 6 7 8 9 10;  5: 0 1 2 3 4 6 7 11 12 13;  6: 0 1 2 3 4 5 8 11 12 14;  7: 0 1 2 3 4 5 11 15 16 17;  8: 0 1 2 3 4 6 12 15 16 18;  9: 0 1 2 3 4 10 11 12 13 14; 10: 0 1 2 3 4 9 11 12 13 14; 11: 5 6 7 9 10 12 13 14 15 16; 12: 5 6 8 9 10 11 13 14 15 16; 13: 5 9 10 11 12 14 15 16 19 20; 14: 6 9 10 11 12 13 15 21 22 23; 15: 7 8 11 12 13 14 19 20 21 24; 16: 7 8 11 12 13 17 18 24 25 26; 17: 7 16 18 19 20 21 22 23 25 26; 18: 8 16 17 19 21 22 23 24 25 26; 19: 13 15 17 18 20 22 23 24 25 26; 20: 13 15 17 19 21 22 23 24 25 26; 21: 14 15 17 18 20 22 23 24 25 26; 22: 14 17 18 19 20 21 23 24 25 26; 23: 14 17 18 19 20 21 22 24 25 26; 24: 15 16 18 19 20 21 22 23 25 26; 25: 16 17 18 19 20 21 22 23 24 26; 26: 16 17 18 19 20 21 22 23 24 25\}

 \item Order 28: \{0: 1 2 3 4 5 6 7 8 9 10;  1: 0 2 3 4 5 6 7 8 9 10;  2: 0 1 3 4 5 6 7 8 9 10;  3: 0 1 2 4 5 6 7 8 9 10;  4: 0 1 2 3 5 6 7 8 9 10;  5: 0 1 2 3 4 6 7 8 9 10;  6: 0 1 2 3 4 5 7 8 11 12;  7: 0 1 2 3 4 5 6 11 12 13;  8: 0 1 2 3 4 5 6 11 13 14;  9: 0 1 2 3 4 5 11 14 15 16; 10: 0 1 2 3 4 5 12 14 17 18; 11: 6 7 8 9 12 13 14 15 16 17; 12: 6 7 10 11 13 14 15 16 17 18; 13: 7 8 11 12 14 15 16 17 18 19; 14: 8 9 10 11 12 13 15 16 17 19; 15: 9 11 12 13 14 17 18 20 21 22; 16: 9 11 12 13 14 20 21 23 24 25; 17: 10 11 12 13 14 15 22 23 26 27; 18: 10 12 13 15 20 22 24 25 26 27; 19: 13 14 20 21 22 23 24 25 26 27; 20: 15 16 18 19 21 23 24 25 26 27; 21: 15 16 19 20 22 23 24 25 26 27; 22: 15 17 18 19 21 23 24 25 26 27; 23: 16 17 19 20 21 22 24 25 26 27; 24: 16 18 19 20 21 22 23 25 26 27; 25: 16 18 19 20 21 22 23 24 26 27; 26: 17 18 19 20 21 22 23 24 25 27; 27: 17 18 19 20 21 22 23 24 25 26\}

 \item Order 29: \{0: 1 2 3 4 5 6 7 8 9 10;  1: 0 2 3 4 5 6 7 8 9 10;  2: 0 1 3 4 5 6 7 8 9 10;  3: 0 1 2 4 5 6 7 8 9 10;  4: 0 1 2 3 5 6 7 8 9 10;  5: 0 1 2 3 4 6 7 8 9 10;  6: 0 1 2 3 4 5 7 8 11 12;  7: 0 1 2 3 4 5 6 11 12 13;  8: 0 1 2 3 4 5 6 11 13 14;  9: 0 1 2 3 4 5 11 14 15 16; 10: 0 1 2 3 4 5 12 14 17 18; 11: 6 7 8 9 12 13 14 15 16 17; 12: 6 7 10 11 13 14 15 16 17 18; 13: 7 8 11 12 14 15 16 17 18 19; 14: 8 9 10 11 12 13 15 16 17 18; 15: 9 11 12 13 14 16 17 19 20 21; 16: 9 11 12 13 14 15 22 23 24 25; 17: 10 11 12 13 14 15 20 22 26 27; 18: 10 12 13 14 20 23 24 26 27 28; 19: 13 15 20 21 22 23 25 26 27 28; 20: 15 17 18 19 21 22 23 24 25 28; 21: 15 19 20 22 23 24 25 26 27 28; 22: 16 17 19 20 21 24 25 26 27 28; 23: 16 18 19 20 21 24 25 26 27 28; 24: 16 18 20 21 22 23 25 26 27 28; 25: 16 19 20 21 22 23 24 26 27 28; 26: 17 18 19 21 22 23 24 25 27 28; 27: 17 18 19 21 22 23 24 25 26 28; 28: 18 19 20 21 22 23 24 25 26 27\}

 \item Order 30: \{0: 1 2 3 4 5 6 7 8 9 10;  1: 0 2 3 4 5 6 7 8 9 10;  2: 0 1 3 4 5 6 7 8 9 10;  3: 0 1 2 4 5 6 7 8 9 10;  4: 0 1 2 3 5 6 7 8 9 10;  5: 0 1 2 3 4 6 7 8 9 10;  6: 0 1 2 3 4 5 7 8 11 12;  7: 0 1 2 3 4 5 6 11 12 13;  8: 0 1 2 3 4 5 6 11 13 14;  9: 0 1 2 3 4 5 11 14 15 16; 10: 0 1 2 3 4 5 12 14 17 18; 11: 6 7 8 9 12 13 14 15 16 17; 12: 6 7 10 11 13 14 15 16 17 18; 13: 7 8 11 12 14 15 16 17 18 19; 14: 8 9 10 11 12 13 15 16 17 18; 15: 9 11 12 13 14 16 17 18 19 20; 16: 9 11 12 13 14 15 17 18 19 21; 17: 10 11 12 13 14 15 16 19 20 22; 18: 10 12 13 14 15 16 21 22 23 24; 19: 13 15 16 17 23 25 26 27 28 29; 20: 15 17 21 23 24 25 26 27 28 29; 21: 16 18 20 22 24 25 26 27 28 29; 22: 17 18 21 23 24 25 26 27 28 29; 23: 18 19 20 22 24 25 26 27 28 29; 24: 18 20 21 22 23 25 26 27 28 29; 25: 19 20 21 22 23 24 26 27 28 29; 26: 19 20 21 22 23 24 25 27 28 29; 27: 19 20 21 22 23 24 25 26 28 29; 28: 19 20 21 22 23 24 25 26 27 29; 29: 19 20 21 22 23 24 25 26 27 28\}

 \item Order 31: \{0: 1 2 3 4 5 6 7 8 9 10;  1: 0 2 3 4 5 6 7 8 9 10;  2: 0 1 3 4 5 6 7 8 9 10;  3: 0 1 2 4 5 6 7 8 9 10;  4: 0 1 2 3 5 6 7 8 9 10;  5: 0 1 2 3 4 6 7 8 9 10;  6: 0 1 2 3 4 5 7 8 11 12;  7: 0 1 2 3 4 5 6 11 12 13;  8: 0 1 2 3 4 5 6 11 13 14;  9: 0 1 2 3 4 5 11 14 15 16; 10: 0 1 2 3 4 5 12 14 17 18; 11: 6 7 8 9 12 13 14 15 16 17; 12: 6 7 10 11 13 14 15 16 17 18; 13: 7 8 11 12 14 15 16 17 18 19; 14: 8 9 10 11 12 13 15 16 17 18; 15: 9 11 12 13 14 16 17 18 19 20; 16: 9 11 12 13 14 15 17 18 19 20; 17: 10 11 12 13 14 15 16 19 21 22; 18: 10 12 13 14 15 16 21 23 24 25; 19: 13 15 16 17 20 23 26 27 28 29; 20: 15 16 19 22 23 24 25 26 27 30; 21: 17 18 22 23 24 25 26 28 29 30; 22: 17 20 21 23 24 26 27 28 29 30; 23: 18 19 20 21 22 25 27 28 29 30; 24: 18 20 21 22 25 26 27 28 29 30; 25: 18 20 21 23 24 26 27 28 29 30; 26: 19 20 21 22 24 25 27 28 29 30; 27: 19 20 22 23 24 25 26 28 29 30; 28: 19 21 22 23 24 25 26 27 29 30; 29: 19 21 22 23 24 25 26 27 28 30; 30: 20 21 22 23 24 25 26 27 28 29\}

 \item Order 32: \{0: 1 2 3 4 5 6 7 8 9 10;  1: 0 2 3 4 5 6 7 8 9 10;  2: 0 1 3 4 5 6 7 8 9 10;  3: 0 1 2 4 5 6 7 8 9 10;  4: 0 1 2 3 5 6 7 8 9 10;  5: 0 1 2 3 4 6 7 8 9 10;  6: 0 1 2 3 4 5 7 8 11 12;  7: 0 1 2 3 4 5 6 11 12 13;  8: 0 1 2 3 4 5 6 11 13 14;  9: 0 1 2 3 4 5 11 14 15 16; 10: 0 1 2 3 4 5 12 14 17 18; 11: 6 7 8 9 12 13 14 15 16 17; 12: 6 7 10 11 13 14 15 16 17 18; 13: 7 8 11 12 14 15 16 17 18 19; 14: 8 9 10 11 12 13 15 16 17 18; 15: 9 11 12 13 14 16 17 18 19 20; 16: 9 11 12 13 14 15 17 18 19 20; 17: 10 11 12 13 14 15 16 18 19 20; 18: 10 12 13 14 15 16 17 19 21 22; 19: 13 15 16 17 18 20 21 23 24 25; 20: 15 16 17 19 21 22 23 26 27 28; 21: 18 19 20 23 26 27 28 29 30 31; 22: 18 20 23 24 25 26 27 29 30 31; 23: 19 20 21 22 24 25 28 29 30 31; 24: 19 22 23 25 26 27 28 29 30 31; 25: 19 22 23 24 26 27 28 29 30 31; 26: 20 21 22 24 25 27 28 29 30 31; 27: 20 21 22 24 25 26 28 29 30 31; 28: 20 21 23 24 25 26 27 29 30 31; 29: 21 22 23 24 25 26 27 28 30 31; 30: 21 22 23 24 25 26 27 28 29 31; 31: 21 22 23 24 25 26 27 28 29 30\}

 \item Order 33: \{0: 1 2 3 4 5 6 7 8 9 10;  1: 0 2 3 4 5 6 7 8 9 10;  2: 0 1 3 4 5 6 7 8 9 10;  3: 0 1 2 4 5 6 7 8 9 10;  4: 0 1 2 3 5 6 7 8 9 10;  5: 0 1 2 3 4 6 7 8 9 10;  6: 0 1 2 3 4 5 7 8 11 12;  7: 0 1 2 3 4 5 6 11 12 13;  8: 0 1 2 3 4 5 6 11 13 14;  9: 0 1 2 3 4 5 11 14 15 16; 10: 0 1 2 3 4 5 12 14 17 18; 11: 6 7 8 9 12 13 14 15 16 17; 12: 6 7 10 11 13 14 15 16 17 18; 13: 7 8 11 12 14 15 16 17 18 19; 14: 8 9 10 11 12 13 15 16 17 18; 15: 9 11 12 13 14 16 17 18 19 20; 16: 9 11 12 13 14 15 17 18 19 20; 17: 10 11 12 13 14 15 16 18 19 20; 18: 10 12 13 14 15 16 17 19 21 22; 19: 13 15 16 17 18 23 24 25 26 27; 20: 15 16 17 21 23 24 28 29 30 31; 21: 18 20 23 25 26 28 29 30 31 32; 22: 18 23 24 25 26 27 28 29 30 32; 23: 19 20 21 22 24 25 27 28 29 32; 24: 19 20 22 23 26 27 28 29 30 31; 25: 19 21 22 23 26 27 28 30 31 32; 26: 19 21 22 24 25 27 29 30 31 32; 27: 19 22 23 24 25 26 29 30 31 32; 28: 20 21 22 23 24 25 29 30 31 32; 29: 20 21 22 23 24 26 27 28 31 32; 30: 20 21 22 24 25 26 27 28 31 32; 31: 20 21 24 25 26 27 28 29 30 32; 32: 21 22 23 25 26 27 28 29 30 31\}

 \item Order 34: \{0: 1 2 3 4 5 6 7 8 9 10;  1: 0 2 3 4 5 6 7 8 9 10;  2: 0 1 3 4 5 6 7 8 9 10;  3: 0 1 2 4 5 6 7 8 9 10;  4: 0 1 2 3 5 6 7 8 9 10;  5: 0 1 2 3 4 6 7 8 9 10;  6: 0 1 2 3 4 5 7 8 11 12;  7: 0 1 2 3 4 5 6 11 12 13;  8: 0 1 2 3 4 5 6 11 13 14;  9: 0 1 2 3 4 5 11 14 15 16; 10: 0 1 2 3 4 5 12 14 17 18; 11: 6 7 8 9 12 13 14 15 16 17; 12: 6 7 10 11 13 14 15 16 17 18; 13: 7 8 11 12 14 15 16 17 18 19; 14: 8 9 10 11 12 13 15 16 17 18; 15: 9 11 12 13 14 16 17 18 19 20; 16: 9 11 12 13 14 15 17 18 19 20; 17: 10 11 12 13 14 15 16 18 19 20; 18: 10 12 13 14 15 16 17 19 20 21; 19: 13 15 16 17 18 20 21 22 23 24; 20: 15 16 17 18 19 21 22 25 26 27; 21: 18 19 20 22 23 24 25 28 29 30; 22: 19 20 21 23 25 26 27 28 31 32; 23: 19 21 22 25 26 27 29 31 32 33; 24: 19 21 25 26 28 29 30 31 32 33; 25: 20 21 22 23 24 26 28 29 30 33; 26: 20 22 23 24 25 27 30 31 32 33; 27: 20 22 23 26 28 29 30 31 32 33; 28: 21 22 24 25 27 29 30 31 32 33; 29: 21 23 24 25 27 28 30 31 32 33; 30: 21 24 25 26 27 28 29 31 32 33; 31: 22 23 24 26 27 28 29 30 32 33; 32: 22 23 24 26 27 28 29 30 31 33; 33: 23 24 25 26 27 28 29 30 31 32\}

\end{itemize}

\subsubsection*{Adjacency lists of 11-regular nut graphs}

\begin{itemize}
\addtolength{\itemsep}{-2mm}
\footnotesize

 \item Order 16: \{0: 1 2 3 4 5 6 7 8 9 10 11;  1: 0 2 3 4 5 6 7 8 9 10 11;  2: 0 1 3 4 5 6 7 8 9 10 12;  3: 0 1 2 4 5 6 7 11 12 13 14;  4: 0 1 2 3 5 6 8 11 13 14 15;  5: 0 1 2 3 4 9 10 11 12 13 15;  6: 0 1 2 3 4 9 11 12 13 14 15;  7: 0 1 2 3 8 9 10 11 12 14 15;  8: 0 1 2 4 7 9 10 12 13 14 15;  9: 0 1 2 5 6 7 8 12 13 14 15; 10: 0 1 2 5 7 8 11 12 13 14 15; 11: 0 1 3 4 5 6 7 10 13 14 15; 12: 2 3 5 6 7 8 9 10 13 14 15; 13: 3 4 5 6 8 9 10 11 12 14 15; 14: 3 4 6 7 8 9 10 11 12 13 15; 15: 4 5 6 7 8 9 10 11 12 13 14\}

 \item Order 18: \{0: 1 2 3 4 5 6 7 8 9 10 11;  1: 0 2 3 4 5 6 7 8 9 10 11;  2: 0 1 3 4 5 6 7 8 9 10 11;  3: 0 1 2 4 5 6 7 8 9 10 11;  4: 0 1 2 3 5 6 7 8 12 13 14;  5: 0 1 2 3 4 6 7 9 12 13 15;  6: 0 1 2 3 4 5 10 12 14 16 17;  7: 0 1 2 3 4 5 11 14 15 16 17;  8: 0 1 2 3 4 12 13 14 15 16 17;  9: 0 1 2 3 5 12 13 14 15 16 17; 10: 0 1 2 3 6 11 12 13 15 16 17; 11: 0 1 2 3 7 10 13 14 15 16 17; 12: 4 5 6 8 9 10 13 14 15 16 17; 13: 4 5 8 9 10 11 12 14 15 16 17; 14: 4 6 7 8 9 11 12 13 15 16 17; 15: 5 7 8 9 10 11 12 13 14 16 17; 16: 6 7 8 9 10 11 12 13 14 15 17; 17: 6 7 8 9 10 11 12 13 14 15 16\}

 \item Order 20: \{0: 1 2 3 4 5 6 7 8 9 10 11;  1: 0 2 3 4 5 6 7 8 9 10 11;  2: 0 1 3 4 5 6 7 8 9 10 11;  3: 0 1 2 4 5 6 7 8 9 10 11;  4: 0 1 2 3 5 6 7 8 9 10 11;  5: 0 1 2 3 4 6 7 8 12 13 14;  6: 0 1 2 3 4 5 12 13 15 16 17;  7: 0 1 2 3 4 5 12 14 15 16 18;  8: 0 1 2 3 4 5 12 15 17 18 19;  9: 0 1 2 3 4 12 13 14 17 18 19; 10: 0 1 2 3 4 13 14 15 16 17 19; 11: 0 1 2 3 4 13 14 15 16 18 19; 12: 5 6 7 8 9 13 15 16 17 18 19; 13: 5 6 9 10 11 12 14 16 17 18 19; 14: 5 7 9 10 11 13 15 16 17 18 19; 15: 6 7 8 10 11 12 14 16 17 18 19; 16: 6 7 10 11 12 13 14 15 17 18 19; 17: 6 8 9 10 12 13 14 15 16 18 19; 18: 7 8 9 11 12 13 14 15 16 17 19; 19: 8 9 10 11 12 13 14 15 16 17 18\}

 \item Order 22: \{0: 1 2 3 4 5 6 7 8 9 10 11;  1: 0 2 3 4 5 6 7 8 9 10 11;  2: 0 1 3 4 5 6 7 8 9 10 11;  3: 0 1 2 4 5 6 7 8 9 10 11;  4: 0 1 2 3 5 6 7 8 9 10 12;  5: 0 1 2 3 4 6 7 8 9 12 13;  6: 0 1 2 3 4 5 11 12 14 15 16;  7: 0 1 2 3 4 5 11 13 14 15 17;  8: 0 1 2 3 4 5 12 13 14 17 18;  9: 0 1 2 3 4 5 12 13 17 19 20; 10: 0 1 2 3 4 11 14 17 18 19 21; 11: 0 1 2 3 6 7 10 15 16 20 21; 12: 4 5 6 8 9 15 16 18 19 20 21; 13: 5 7 8 9 14 15 16 18 19 20 21; 14: 6 7 8 10 13 16 17 18 19 20 21; 15: 6 7 11 12 13 16 17 18 19 20 21; 16: 6 11 12 13 14 15 17 18 19 20 21; 17: 7 8 9 10 14 15 16 18 19 20 21; 18: 8 10 12 13 14 15 16 17 19 20 21; 19: 9 10 12 13 14 15 16 17 18 20 21; 20: 9 11 12 13 14 15 16 17 18 19 21; 21: 10 11 12 13 14 15 16 17 18 19 20\}

 \item Order 24: \{0: 1 2 3 4 5 6 7 8 9 10 11;  1: 0 2 3 4 5 6 7 8 9 10 11;  2: 0 1 3 4 5 6 7 8 9 10 11;  3: 0 1 2 4 5 6 7 8 9 10 11;  4: 0 1 2 3 5 6 7 8 9 10 11;  5: 0 1 2 3 4 6 7 8 9 10 12;  6: 0 1 2 3 4 5 7 8 9 13 14;  7: 0 1 2 3 4 5 6 10 15 16 17;  8: 0 1 2 3 4 5 6 10 15 16 18;  9: 0 1 2 3 4 5 6 10 15 19 20; 10: 0 1 2 3 4 5 7 8 9 13 21; 11: 0 1 2 3 4 13 14 17 18 22 23; 12: 5 13 14 15 16 17 19 20 21 22 23; 13: 6 10 11 12 17 18 19 20 21 22 23; 14: 6 11 12 15 16 17 18 19 20 21 22; 15: 7 8 9 12 14 16 17 19 20 22 23; 16: 7 8 12 14 15 17 18 19 21 22 23; 17: 7 11 12 13 14 15 16 18 20 21 23; 18: 8 11 13 14 16 17 19 20 21 22 23; 19: 9 12 13 14 15 16 18 20 21 22 23; 20: 9 12 13 14 15 17 18 19 21 22 23; 21: 10 12 13 14 16 17 18 19 20 22 23; 22: 11 12 13 14 15 16 18 19 20 21 23; 23: 11 12 13 15 16 17 18 19 20 21 22\}

 \item Order 26: \{0: 1 2 3 4 5 6 7 8 9 10 11;  1: 0 2 3 4 5 6 7 8 9 10 11;  2: 0 1 3 4 5 6 7 8 9 10 11;  3: 0 1 2 4 5 6 7 8 9 10 11;  4: 0 1 2 3 5 6 7 8 9 10 11;  5: 0 1 2 3 4 6 7 8 9 12 13;  6: 0 1 2 3 4 5 10 11 14 15 16;  7: 0 1 2 3 4 5 10 14 15 16 17;  8: 0 1 2 3 4 5 10 14 15 18 19;  9: 0 1 2 3 4 5 18 19 20 21 22; 10: 0 1 2 3 4 6 7 8 12 13 14; 11: 0 1 2 3 4 6 12 13 14 15 16; 12: 5 10 11 13 14 15 16 17 18 19 20; 13: 5 10 11 12 14 15 18 21 22 23 24; 14: 6 7 8 10 11 12 13 19 20 23 25; 15: 6 7 8 11 12 13 17 21 23 24 25; 16: 6 7 11 12 17 19 20 22 23 24 25; 17: 7 12 15 16 18 20 21 22 23 24 25; 18: 8 9 12 13 17 19 20 21 22 24 25; 19: 8 9 12 14 16 18 21 22 23 24 25; 20: 9 12 14 16 17 18 21 22 23 24 25; 21: 9 13 15 17 18 19 20 22 23 24 25; 22: 9 13 16 17 18 19 20 21 23 24 25; 23: 13 14 15 16 17 19 20 21 22 24 25; 24: 13 15 16 17 18 19 20 21 22 23 25; 25: 14 15 16 17 18 19 20 21 22 23 24\}

 \item Order 28: \{0: 1 2 3 4 5 6 7 8 9 10 11;  1: 0 2 3 4 5 6 7 8 9 10 11;  2: 0 1 3 4 5 6 7 8 9 10 11;  3: 0 1 2 4 5 6 7 8 9 10 11;  4: 0 1 2 3 5 6 7 8 9 10 11;  5: 0 1 2 3 4 6 7 8 9 12 13;  6: 0 1 2 3 4 5 10 11 14 15 16;  7: 0 1 2 3 4 5 10 12 17 18 19;  8: 0 1 2 3 4 5 10 14 15 20 21;  9: 0 1 2 3 4 5 11 14 15 16 20; 10: 0 1 2 3 4 6 7 8 12 13 14; 11: 0 1 2 3 4 6 9 12 13 14 15; 12: 5 7 10 11 13 14 15 16 17 18 19; 13: 5 10 11 12 14 15 16 17 18 19 20; 14: 6 8 9 10 11 12 13 16 17 22 23; 15: 6 8 9 11 12 13 18 22 24 25 26; 16: 6 9 12 13 14 19 22 23 24 25 27; 17: 7 12 13 14 21 22 23 24 25 26 27; 18: 7 12 13 15 20 21 22 24 25 26 27; 19: 7 12 13 16 21 22 23 24 25 26 27; 20: 8 9 13 18 21 22 23 24 25 26 27; 21: 8 17 18 19 20 22 23 24 25 26 27; 22: 14 15 16 17 18 19 20 21 23 26 27; 23: 14 16 17 19 20 21 22 24 25 26 27; 24: 15 16 17 18 19 20 21 23 25 26 27; 25: 15 16 17 18 19 20 21 23 24 26 27; 26: 15 17 18 19 20 21 22 23 24 25 27; 27: 16 17 18 19 20 21 22 23 24 25 26\}

 \item Order 30: \{0: 1 2 3 4 5 6 7 8 9 10 11;  1: 0 2 3 4 5 6 7 8 9 10 11;  2: 0 1 3 4 5 6 7 8 9 10 11;  3: 0 1 2 4 5 6 7 8 9 10 11;  4: 0 1 2 3 5 6 7 8 9 10 11;  5: 0 1 2 3 4 6 7 8 9 12 13;  6: 0 1 2 3 4 5 10 11 14 15 16;  7: 0 1 2 3 4 5 10 12 17 18 19;  8: 0 1 2 3 4 5 10 14 15 16 20;  9: 0 1 2 3 4 5 13 21 22 23 24; 10: 0 1 2 3 4 6 7 8 12 13 14; 11: 0 1 2 3 4 6 12 13 14 15 16; 12: 5 7 10 11 13 14 15 16 17 18 19; 13: 5 9 10 11 12 14 15 16 17 18 19; 14: 6 8 10 11 12 13 15 16 17 18 19; 15: 6 8 11 12 13 14 16 17 18 19 20; 16: 6 8 11 12 13 14 15 17 21 22 25; 17: 7 12 13 14 15 16 18 20 23 25 26; 18: 7 12 13 14 15 17 21 25 27 28 29; 19: 7 12 13 14 15 23 24 26 27 28 29; 20: 8 15 17 22 23 24 25 26 27 28 29; 21: 9 16 18 22 23 24 25 26 27 28 29; 22: 9 16 20 21 23 24 25 26 27 28 29; 23: 9 17 19 20 21 22 24 26 27 28 29; 24: 9 19 20 21 22 23 25 26 27 28 29; 25: 16 17 18 20 21 22 24 26 27 28 29; 26: 17 19 20 21 22 23 24 25 27 28 29; 27: 18 19 20 21 22 23 24 25 26 28 29; 28: 18 19 20 21 22 23 24 25 26 27 29; 29: 18 19 20 21 22 23 24 25 26 27 28\}

 \item Order 32: \{0: 1 2 3 4 5 6 7 8 9 10 11;  1: 0 2 3 4 5 6 7 8 9 10 11;  2: 0 1 3 4 5 6 7 8 9 10 11;  3: 0 1 2 4 5 6 7 8 9 10 11;  4: 0 1 2 3 5 6 7 8 9 10 11;  5: 0 1 2 3 4 6 7 8 9 12 13;  6: 0 1 2 3 4 5 10 11 14 15 16;  7: 0 1 2 3 4 5 10 12 17 18 19;  8: 0 1 2 3 4 5 10 14 15 16 20;  9: 0 1 2 3 4 5 12 17 18 21 22; 10: 0 1 2 3 4 6 7 8 12 13 14; 11: 0 1 2 3 4 6 12 13 14 15 16; 12: 5 7 9 10 11 13 14 15 16 17 18; 13: 5 10 11 12 14 15 16 17 18 19 20; 14: 6 8 10 11 12 13 15 16 17 18 19; 15: 6 8 11 12 13 14 16 17 18 19 20; 16: 6 8 11 12 13 14 15 17 18 19 20; 17: 7 9 12 13 14 15 16 18 19 20 23; 18: 7 9 12 13 14 15 16 17 21 24 25; 19: 7 13 14 15 16 17 23 24 25 26 27; 20: 8 13 15 16 17 22 23 28 29 30 31; 21: 9 18 23 24 25 26 27 28 29 30 31; 22: 9 20 23 24 25 26 27 28 29 30 31; 23: 17 19 20 21 22 26 27 28 29 30 31; 24: 18 19 21 22 25 26 27 28 29 30 31; 25: 18 19 21 22 24 26 27 28 29 30 31; 26: 19 21 22 23 24 25 27 28 29 30 31; 27: 19 21 22 23 24 25 26 28 29 30 31; 28: 20 21 22 23 24 25 26 27 29 30 31; 29: 20 21 22 23 24 25 26 27 28 30 31; 30: 20 21 22 23 24 25 26 27 28 29 31; 31: 20 21 22 23 24 25 26 27 28 29 30\}

 \item Order 34: \{0: 1 2 3 4 5 6 7 8 9 10 11;  1: 0 2 3 4 5 6 7 8 9 10 11;  2: 0 1 3 4 5 6 7 8 9 10 11;  3: 0 1 2 4 5 6 7 8 9 10 11;  4: 0 1 2 3 5 6 7 8 9 10 11;  5: 0 1 2 3 4 6 7 8 9 12 13;  6: 0 1 2 3 4 5 10 11 14 15 16;  7: 0 1 2 3 4 5 10 12 17 18 19;  8: 0 1 2 3 4 5 10 14 15 16 20;  9: 0 1 2 3 4 5 12 17 18 20 21; 10: 0 1 2 3 4 6 7 8 12 13 14; 11: 0 1 2 3 4 6 12 13 14 15 16; 12: 5 7 9 10 11 13 14 15 16 17 18; 13: 5 10 11 12 14 15 16 17 18 19 20; 14: 6 8 10 11 12 13 15 16 17 18 19; 15: 6 8 11 12 13 14 16 17 18 19 20; 16: 6 8 11 12 13 14 15 17 18 19 20; 17: 7 9 12 13 14 15 16 18 19 20 21; 18: 7 9 12 13 14 15 16 17 19 20 22; 19: 7 13 14 15 16 17 18 20 21 23 24; 20: 8 9 13 15 16 17 18 19 25 26 27; 21: 9 17 19 22 23 24 25 28 29 30 31; 22: 18 21 23 24 25 26 27 28 29 32 33; 23: 19 21 22 25 26 27 28 30 31 32 33; 24: 19 21 22 25 26 28 29 30 31 32 33; 25: 20 21 22 23 24 27 29 30 31 32 33; 26: 20 22 23 24 27 28 29 30 31 32 33; 27: 20 22 23 25 26 28 29 30 31 32 33; 28: 21 22 23 24 26 27 29 30 31 32 33; 29: 21 22 24 25 26 27 28 30 31 32 33; 30: 21 23 24 25 26 27 28 29 31 32 33; 31: 21 23 24 25 26 27 28 29 30 32 33; 32: 22 23 24 25 26 27 28 29 30 31 33; 33: 22 23 24 25 26 27 28 29 30 31 32\}

 \item Order 36: \{0: 1 2 3 4 5 6 7 8 9 10 11;  1: 0 2 3 4 5 6 7 8 9 10 11;  2: 0 1 3 4 5 6 7 8 9 10 11;  3: 0 1 2 4 5 6 7 8 9 10 11;  4: 0 1 2 3 5 6 7 8 9 10 11;  5: 0 1 2 3 4 6 7 8 9 12 13;  6: 0 1 2 3 4 5 10 11 14 15 16;  7: 0 1 2 3 4 5 10 12 17 18 19;  8: 0 1 2 3 4 5 10 14 15 16 20;  9: 0 1 2 3 4 5 12 17 18 20 21; 10: 0 1 2 3 4 6 7 8 12 13 14; 11: 0 1 2 3 4 6 12 13 14 15 16; 12: 5 7 9 10 11 13 14 15 16 17 18; 13: 5 10 11 12 14 15 16 17 18 19 20; 14: 6 8 10 11 12 13 15 16 17 18 19; 15: 6 8 11 12 13 14 16 17 18 19 20; 16: 6 8 11 12 13 14 15 17 18 19 20; 17: 7 9 12 13 14 15 16 18 19 20 21; 18: 7 9 12 13 14 15 16 17 19 20 21; 19: 7 13 14 15 16 17 18 20 21 22 23; 20: 8 9 13 15 16 17 18 19 21 22 24; 21: 9 17 18 19 20 22 23 24 25 26 27; 22: 19 20 21 23 24 25 26 27 28 29 30; 23: 19 21 22 25 28 29 31 32 33 34 35; 24: 20 21 22 26 28 29 31 32 33 34 35; 25: 21 22 23 27 28 29 30 31 32 33 34; 26: 21 22 24 27 28 29 30 31 32 33 35; 27: 21 22 25 26 28 30 31 32 33 34 35; 28: 22 23 24 25 26 27 29 30 31 34 35; 29: 22 23 24 25 26 28 30 32 33 34 35; 30: 22 25 26 27 28 29 31 32 33 34 35; 31: 23 24 25 26 27 28 30 32 33 34 35; 32: 23 24 25 26 27 29 30 31 33 34 35; 33: 23 24 25 26 27 29 30 31 32 34 35; 34: 23 24 25 27 28 29 30 31 32 33 35; 35: 23 24 26 27 28 29 30 31 32 33 34\}
\end{itemize}

\end{document}